\documentclass[a4paper,10pt]{amsart}
\usepackage{amsmath,amssymb,amsthm}
\usepackage{epsfig}
\usepackage{graphicx}
\newcommand{\geo}{\operatorname{geo}}

\newcommand{\CM}{\overline{\operatorname{\mathcal{M}}}}

\newcommand{\IM}{\operatorname{\mathcal{M}}}

\newcommand{\LL}{\operatorname{\mathcal{L}}}

\newcommand{\PP}{\operatorname{\mathfrak{P}}}
\newcommand{\WW}{\operatorname{\mathfrak{W}}}

\newcommand{\Si}{\dot{S}}

\newcommand{\CP}{\mathbb{C}\mathbb{P}^1}

\newcommand{\CR}{\bar{\partial}}
\newcommand{\uz}{\underline{z}}
\newcommand{\Mph}{M_{\phi}}

\newcommand{\Symp}{\operatorname{Symp}}

\newcommand{\ev}{\operatorname{ev}}

\newcommand{\CZ}{\operatorname{CZ}}

\newcommand{\diag}{\operatorname{diag}}

\newcommand{\del}{\partial}

\newcommand{\RS}{\IR \times S^1}
\newcommand{\Ju}{\underline{J}}

\newcommand{\IC}{\operatorname{\mathbb{C}}}

\newcommand{\IR}{\operatorname{\mathbb{R}}}
\newcommand{\IN}{\operatorname{\mathbb{N}}}
\newcommand{\ID}{\operatorname{\mathbf{D}}}
\newcommand{\IH}{\operatorname{\mathbf{H}}}

\newcommand{\Ih}{\operatorname{\mathbf{h}}}
\newcommand{\Ig}{\operatorname{\mathbf{g}}}

\newcommand{\IF}{\operatorname{\mathbf{F}}}
\newcommand{\If}{\operatorname{\mathbf{f}}}

\newtheorem{theorem}{Theorem}[section]
\newtheorem{proposition}[theorem]{Proposition}

\newtheorem{corollary}[theorem]{Corollary}
\newtheorem{remark}[theorem]{Remark}
\newtheorem{ex}[theorem]{Example}
\newenvironment{example}{\begin{ex}\rm}{\qee\end{ex}}
\newcommand{\qee}{\mbox{\hspace{0.2mm}}\hfill$\triangle$}

\title{Topological recursion relations in \\non-equivariant cylindrical contact homology}
\author{Oliver Fabert and Paolo Rossi}
\begin{document}

\begin{abstract}
It was pointed out by Eliashberg in his ICM 2006 plenary talk that the integrable systems of rational Gromov-Witten theory very naturally appear in the rich algebraic formalism of symplectic field theory (SFT). Carefully generalizing the definition of gravitational descendants from Gromov-Witten theory to SFT, one can assign to every contact manifold a Hamiltonian system with symmetries on SFT homology and the question of its integrability arises. While we have shown how the well-known string, dilaton and divisor equations translate from Gromov-Witten theory to SFT, the next step is to show how genus-zero topological recursion translates to SFT. Compatible with the example of SFT of closed geodesics, it turns out that the corresponding localization theorem requires a non-equivariant version of SFT, which is generated by parametrized instead of unparametrized closed Reeb orbits. Since this non-equivariant version is so far only defined for cylindrical contact homology, we restrict ourselves to this special case. As an important result we show that, as in rational Gromov-Witten theory, all descendant invariants can be computed from primary invariants, i.e. without descendants.
\end{abstract}

\maketitle

\tableofcontents

\markboth{O. Fabert and P. Rossi}{Topological recursion in cylindrical homology} 

\section{Introduction}

Symplectic field theory (SFT), introduced by H. Hofer, A. Givental and Y. Eliashberg in 2000 ([EGH]), is a very large 
project and can be viewed as a topological quantum field theory approach to Gromov-Witten theory. Besides providing a 
unified view on established pseudoholomorphic curve theories like symplectic Floer homology, contact homology and 
Gromov-Witten theory, it leads to numerous new applications and opens new routes yet to be explored. \\
 
While symplectic field theory leads to algebraic invariants with very rich algebraic structures, it was pointed out by Eliashberg in his ICM 2006 plenary talk ([E]) that the integrable systems of rational Gromov-Witten theory very naturally appear in rational symplectic field theory by using the link between the rational symplectic field theory of prequantization spaces in the Morse-Bott version and the rational Gromov-Witten potential of the underlying symplectic manifold, see the recent papers \cite{R1}, \cite{R2} by the second author. Indeed, after introducing gravitational descendants as in Gromov-Witten theory, it is precisely the rich algebraic formalism of SFT with its Weyl and Poisson structures that provides a natural link between symplectic field theory and (quantum) integrable systems. \\ 
 
On the other hand, carefully defining a generalization of gravitational descendants and adding them to the picture, the first author has shown in \cite{F2} that one can assign to every contact manifold an infinite sequence of commuting Hamiltonian systems on 
SFT homology and the question of their integrability arises. For this it is important to fully understand the algebraic structure of 
gravitational descendants in SFT. \\

While it is well-known that in Gromov-Witten theory the topological meaning of gravitational descendants leads to new differential 
equations for the Gromov-Witten potential, in this paper we want to proceed with our project of understanding how these rich algebraic structures carry over from Gromov-Witten theory to symplectic field theory.  While we have already shown in \cite{FR} how the well-known string, dilaton and divisor equations translate from Gromov-Witten theory to SFT, as a next step we want to show how classical genus-zero topological recursion generalizes to symplectic field theory. \\

Although this is a first concrete step in the study of integrability of the Hamiltonian systems of SFT, notice that topological recursion relations in the forms we study here might not be enough to answer the question of integrability: the Hamiltonian systems arising from SFT are, a priori, much more general than those associated with Gromov-Witten invariants, involving in particular more than just local functionals (see \cite{R2}), and topological recursion relations, even together with string, dilaton and divisor equations, might not yet restrictive enough to grant complete control over the algebra of commuting Hamiltonians. They seem however to give an affirmative answer to the fundamental question of the reconstructability of the gravitational descendants from the primary invariants (i.e. without descendants) in genus $0$.\\

From the computation of the SFT of a Reeb orbit with descendants in \cite{F2} it can be seen that the genus-zero topological recursion requires a non-equivariant version of SFT, which is generated by parametrized instead of unparametrized Reeb orbits. The definition of this non-equivariant version of SFT is currently a very active field of research and related to the work of Bourgeois and Oancea in \cite{BO}, where a Gysin-type spectral sequence relating linearized contact homology (a slight generalization of cylindrical contact homology depending on a symplectic filling) and symplectic homology of this filling is established by viewing the one as the 
(non-)equivariant version of the other. \\

Since the topological recursion relation is already interesting in the case of cylindrical contact homology and the non-equivariant version of it is already understood, in this first paper on topological recursion we restrict ourselves to cylindrical contact homology, i.e. study the algebraic structure of gravitational descendants only for this special case. \\

This paper is organized as follows: While in section two we review the most important definitions and results about SFT with gravitational descendants and its relation with integrable systems in \cite{F2} and \cite{FR} (including short discussions of how to interpret the SFT homology algebra as a Poisson algebra of functions on a singular phase space and how to generalize the notion of local functionals from Gromov-Witten theory to general contact manifolds), in section three we first show, as a motivation for our main result, how the topological recursion relations in Gromov-Witten theory carry over to symplectic Floer theory. Since this example suggests that the localization theorem for gravitational descendants needs a non-equivariant version of cylindrical contact homology which, similar to symplectic Floer homology, is generated by parametrized instead of unparametrized closed Reeb orbits, we recall in section four the definition of non-equivariant cylindrical homology from \cite{BO} and prove the topological recursion relations in the non-equivariant situation. Finally, in section five we discuss two important applications of our main result. First we show how the topological recursion formulas carry over from the non-equivariant to the equivariant situation and use this result to show that, as in rational Gromov-Witten theory, all descendant invariants can be computed from primary invariants, that is, without descendants. After this we show that our results can be further used to define an action of the (quantum) cohomology on non-equivariant cylindrical homology similar to the corresponding action on symplectic Floer homology defined in \cite{PSS}. At the end we show that in the Floer case of SFT we just get back the topological recursion relations in Floer homology from section three and that the action of quantum cohomology on non-equivariant homology splits and agrees with the action on Floer homology as defined in \cite{PSS}.\\

While the work on this paper began when both authors were members of the Mathematical Sciences Research Institute (MSRI) in Berkeley, most of the work was conducted when the first author was a postdoc at the Max Planck Institute (MPI) for Mathematics in the Sciences in Germany and the second author was a FSMP postdoc at the Institut de Mathematiques de Jussieu, Paris VI. They want to thank the institutes for their hospitality and their great working environment. Further they would like to thank Y. Eliashberg and D. Zvonkine for useful discussions.

\vspace{0.5cm}

\section{Symplectic field theory with gravitational descendants}

\subsection{Symplectic Field Theory}
Symplectic field theory (SFT) is a very large project, initiated by Y. Eliashberg,
A. Givental and H. Hofer in their paper \cite{EGH}, designed to describe in a unified way 
the theory of pseudoholomorphic curves in symplectic and contact topology. 
Besides providing a unified view on well-known theories like symplectic Floer 
homology and Gromov-Witten theory, it shows how to assign algebraic invariants 
to closed manifolds with a stable Hamiltonian structure. \\

Following \cite{BEHWZ} a Hamiltonian structure on a closed $(2m-1)$-dimensional manifold $V$ is a closed two-form $\omega$ 
on $V$, which is maximally nondegenerate in the sense that $\ker\omega=\{v\in TV:\omega(v,\cdot)=0\}$ 
is a one-dimensional distribution. The Hamiltonian structure is required to be stable in the sense that there exists a 
one-form $\lambda$ on $V$ such that $\ker\omega\subset\ker d\lambda$ and $\lambda(v)\neq 0$ for all $v\in\ker\omega-\{0\}$. 
Any stable Hamiltonian structure $(\omega,\lambda)$ defines a symplectic hyperplane distribution $(\xi=\ker\lambda,\omega_{\xi})$, 
where $\omega_{\xi}$ is the restriction of $\omega$, and 
a vector field $R$ on $V$ by requiring $R\in\ker\omega$ and $\lambda(R)=1$, which is called the Reeb vector field of the 
stable Hamiltonian structure. Examples for closed manifolds $V$ with a stable Hamiltonian structure $(\omega,\lambda)$ 
are contact manifolds, symplectic mapping tori and principal circle bundles over symplectic manifolds ([BEHWZ]): \\
 
First observe that when $\lambda$ is a contact form on $V$, it is easy to check that $(\omega:=d\lambda,\lambda)$ is a stable 
Hamiltonian structure and the symplectic hyperplane distribution agrees with the contact structure. 
For the other two cases, let $(M,\omega_M)$ be a symplectic manifold. Then every principal circle bundle $S^1\to V\to M$ and 
every symplectic mapping torus $M\to V\to S^1$, i.e. $V=\Mph=\IR\times M/\{(t,p)\sim(t+1,\phi(p))\}$ for 
$\phi\in\Symp(M,\omega)$ also carries a stable Hamiltonian structure. For the circle bundle the Hamiltonian 
structure is given by the pullback $\pi^*\omega$ under the bundle projection and we can choose as one-form $\lambda$ any $S^1$-connection form. 
On the other hand, the stable 
Hamiltonian structure on the mapping torus $V=\Mph$ is given by lifting the symplectic form to $\omega\in\Omega^2(\Mph)$ via the natural 
flat connection $TV=TS^1\oplus TM$ and setting $\lambda=dt$ for the natural $S^1$-coordinate $t$ on $\Mph$. 
While in the mapping torus case $\xi$ is always integrable, in the circle bundle case the hyperplane distribution $\xi$ may be integrable or 
non-integrable, even contact. \\ 

Symplectic field theory assigns algebraic invariants to closed manifolds $V$ with a stable Hamiltonian structure. 
The invariants are defined by counting $\Ju$-holomorphic curves in $\IR\times V$ with finite energy, 
where the underlying closed Riemann surfaces are explicitly allowed to have punctures, i.e. single points are removed.    
The almost complex structure $\Ju$ on the cylindrical manifold $\IR\times V$ is required to be cylindrical in the sense that it is  
$\IR$-independent, links the two natural vector fields on $\IR\times V$, namely the Reeb vector field 
$R$ and the $\IR$-direction $\del_s$, by $\Ju\del_s=R$, and turns the symplectic hyperplane distribution on $V$ into a complex subbundle of $TV$, $\xi=TV\cap \Ju TV$. It follows that a cylindrical almost complex structure $\Ju$ on $\IR\times V$ is determined by its restriction $\Ju_{\xi}$ to $\xi\subset TV$, which is required to be $\omega_{\xi}$-compatible in the sense that $\omega_{\xi}(\cdot,\Ju_{\xi}\cdot)$ defines a metric on $\xi$. Note that such almost complex structures $\Ju$ are called compatible with the stable Hamiltonian structure and that the set of these almost complex structures is non-empty and contractible. \\

We assume that the stable Hamiltonian structure is Morse in the sense that all closed orbits of the 
Reeb vector field are nondegenerate in the sense of \cite{BEHWZ}; in particular, the set 
of closed Reeb orbits is discrete. Further it is shown in \cite{BEHWZ} that all $\Ju$-holomorphic curves in $\IR\times V$ with finite energy are asymptotically cylindrical over collections of Reeb orbits $\Gamma^{\pm}=\{\gamma^{\pm}_1,\ldots,
\gamma^{\pm}_{n^{\pm}}\}$ as the $\IR$-factor tends to $\pm\infty$. We denote by $\CM_{g,r,A}(\Gamma^+,\Gamma^-)/\IR$ the corresponding compactified moduli space of genus $g$ curves with $r$ additional marked points representing the absolute homology class $A\in H_2(V)$ using a choice of spanning surfaces ([BEHWZ],[EGH]). After choosing abstract perturbations using polyfolds following \cite{HWZ}, we get that 
$\CM_{g,r,A}(\Gamma^+,\Gamma^-)$ is a (weighted branched) manifold with corners of dimension 
equal to the Fredholm index of the Cauchy-Riemann operator for $\Ju$. 
{\it Note that as in \cite{F2} and \cite{FR} we will not discuss transversality for the Cauchy-Riemann operator but just refer to the upcoming 
papers on polyfolds by H. Hofer and his co-workers.} \\  
 
Let us now briefly introduce the algebraic formalism of SFT as described in \cite{EGH}: \\
 
Recall that a multiply-covered Reeb orbit $\gamma^k$ is called bad if 
$\CZ(\gamma^k)\neq\CZ(\gamma)\mod 2$, where $\CZ(\gamma)$ denotes the 
Conley-Zehnder index of $\gamma$. Calling a Reeb orbit $\gamma$ {\it good} if it is not bad we assign to every 
good Reeb orbit $\gamma$ two formal graded variables $p_{\gamma},q_{\gamma}$ with grading 
\begin{equation*} 
|p_{\gamma}|=m-3-\CZ(\gamma),|q_{\gamma}|=m-3+\CZ(\gamma) 
\end{equation*} 
when $\dim V = 2m-1$. Assuming we have chosen a basis $A_0,\ldots,A_N$ of $H_2(V)$, we assign to every $A_i$ a formal 
variables $z_i$ with grading $|z_i|=- 2 c_1(A_i)$. In order to include higher-dimensional moduli spaces we further assume that a string 
of closed (homogeneous) differential forms $\Theta=(\theta_1,\ldots,\theta_N)$ on $V$ is chosen and assign to 
every $\theta_{\alpha}\in\Omega^*(V)$ a formal variables $t^{\alpha}$ 
with grading
\begin{equation*} |t^{\alpha}|=2 -\deg\theta_{\alpha}. \end{equation*}  
Finally, let $\hbar$ be another formal variable of degree $|\hbar|=2(m-3)$. \\

Let $\WW$ be the graded Weyl algebra over $\IC$ of power series in the variables 
$\hbar,p_{\gamma}$ and $t_i$ with coefficients which are polynomials in the 
variables $q_{\gamma}$ and Laurent series in $z_n$, which is equipped with the associative product $\star$ in 
which all variables super-commute according to their grading except for the 
variables $p_{\gamma}$, $q_{\gamma}$ corresponding to the same Reeb orbit $\gamma$, 
\begin{equation*} [p_{\gamma},q_{\gamma}] = 
                  p_{\gamma}\star q_{\gamma} -(-1)^{|p_{\gamma}||q_{\gamma}|} 
                  q_{\gamma}\star p_{\gamma} = \kappa_{\gamma}\hbar.
\end{equation*}
($\kappa_{\gamma}$ denotes the multiplicity of $\gamma$.) Since it is shown in \cite{EGH} that the bracket 
of two elements in $\WW$ gives an element in $\hbar\WW$, it follows that we get a bracket on the module 
$\hbar^{-1}\WW$. Following \cite{EGH} we further introduce 
the Poisson algebra $\PP$ of formal power series in the variables $p_{\gamma}$ and $t_i$ with   
coefficients which are polynomials in the variables $q_{\gamma}$ and Laurent series in $z_n$ with Poisson bracket given by 
\begin{equation*} 
 \{f,g\} = \sum_{\gamma}\kappa_{\gamma}\Bigl(\frac{\del f}{\del p_{\gamma}}\frac{\del g}{\del q_{\gamma}} -
                          (-1)^{|f||g|}\frac{\del g}{\del p_{\gamma}}\frac{\del f}{\del q_{\gamma}}\Bigr).       
\end{equation*}

As in Gromov-Witten theory we want to organize all moduli spaces $\CM_{g,r,A}(\Gamma^+,\Gamma^-)$
into a generating function $\IH\in\hbar^{-1}\WW$, called {\it Hamiltonian}. In order to include also higher-dimensional 
moduli spaces, in \cite{EGH} the authors follow the approach in Gromov-Witten theory to integrate the chosen differential forms 
$\theta_{\alpha}$ over the moduli spaces after pulling them back under the evaluation map from target manifold $V$. 
The Hamiltonian $\IH$ is then defined by
\begin{equation*}
 \IH = \sum_{\Gamma^+,\Gamma^-} \int_{\CM_{g,r,A}(\Gamma^+,\Gamma^-)/\IR}
 \ev_1^*\theta_{\alpha_1}\wedge\ldots\wedge\ev_r^*\theta_{\alpha_r}\; \hbar^{g-1}t^\alpha p^{\Gamma^+}q^{\Gamma^-}z^d
\end{equation*}
with $t^{I}=t^{\alpha_1}\ldots t^{\alpha_r}$, $p^{\Gamma^+}=p_{\gamma^+_1}\ldots p_{\gamma^+_{n^+}}$, 
$q^{\Gamma^-}=q_{\gamma^-_1}\ldots q_{\gamma^-_{n^-}}$ and $z^d = z_0^{d_0} \cdot \ldots \cdot z_N^{d_N}$. 
Expanding 
\begin{equation*} \IH=\hbar^{-1}\sum_g \IH_g \hbar^g \end{equation*} 
we further get a rational Hamiltonian $\Ih=\IH_0\in\PP$, which counts only curves with genus zero. \\

While the Hamiltonian $\IH$ explicitly depends on the chosen contact form, the cylindrical almost complex structure, 
the differential forms and abstract polyfold perturbations making all moduli spaces regular, it is outlined in \cite{EGH} 
how to construct algebraic invariants, which just depend on the contact structure and the cohomology classes of the 
differential forms.
 
\vspace{0.5cm}

\subsection{Gravitational descendants}
In complete analogy to Gromov-Witten theory we can introduce $r$ tautological line 
bundles $\LL_1,\ldots,\LL_r$ over each moduli space $\CM_r=\CM_{g,r,A}(\Gamma^+,\Gamma^-)/\IR$ , where the fibre of $\LL_i$ 
over a punctured curve $(u,\Si)\in\CM_r$ is again given 
by the cotangent line to the underlying, possibly unstable nodal Riemann surface (without ghost components) at the 
$i$.th marked point and which again formally can be defined as the pull-back of the vertical cotangent line 
bundle of $\pi: \CM_{r+1}\to\CM_r$ under the canonical section $\sigma_i: \CM_r\to\CM_{r+1}$ mapping to the $i$.th marked 
point in the fibre. Note again that while the vertical cotangent line bundle is rather a sheaf (the dualizing sheaf) than a true bundle since 
it becomes singular at the nodes in the fibres, the pull-backs under the canonical sections are still true line bundles 
as the marked points are different from the nodes and hence these sections avoid the singular loci. \\

While in Gromov-Witten theory the gravitational descendants were defined by integrating powers of the first Chern class 
of the tautological line bundle over the moduli space, which by Poincare duality corresponds to counting common zeroes of 
sections in this bundle, in symplectic field theory, more generally every holomorphic curves theory where curves with 
punctures and/or boundary are considered, we are faced with the problem that the moduli spaces generically have 
codimension-one boundary, so that the count of zeroes of sections in general depends on the chosen sections in the 
boundary. It follows that the integration of the first Chern class of the tautological line bundle over a single moduli 
space has to be replaced by a construction involving all moduli space at once. Note that this is similar to the choice of 
coherent abstract perturbations for the moduli spaces in symplectic field theory in order to achieve transversality for 
the Cauchy-Riemann operator. \\

Keeping the interpretation of descendants as common zero sets of sections in powers of the 
tautological line bundles, the first author defined in his paper \cite{F2} the notion of {\it coherent collections of sections} 
$(s)$ in the tautological line bundles over all moduli spaces, which just formalizes how the sections chosen for the 
lower-dimensional moduli spaces should affect the section chosen for a moduli spaces on its boundary. Based on this he then 
defined {\it descendants of moduli spaces} $\CM^j\subset\CM$, which were obtained inductively as zero sets of these coherent 
collections of sections $(s_j)$ in the tautological line bundles over the descendant moduli spaces $\CM^{j-1}\subset\CM$. \\

So far we have only considered the case with one additional marked point. On the other hand, as already outlined in \cite{F2}, 
the general case with $r$ additional marked points is just notationally more involved. Indeed, we can 
easily define for every moduli space $\CM_r=\CM_{g,r,A}(\Gamma^+,\Gamma^-)/\IR$ with $r$ additional marked points and every 
$r$-tuple of natural numbers $(j_1,\ldots,j_r)$ descendants $\CM^{(j_1,\ldots,j_r)}_r\subset\CM_r$ by setting
\begin{equation*} \CM^{(j_1,\ldots,j_r)}_r = \CM^{(j_1,0,\ldots,0)}_r\cap \ldots \cap \CM^{(0,\ldots,0,j_r)}_r, \end{equation*}
where the descendant moduli spaces $\CM^{(0,\ldots,0,j_i,0,\ldots,0)}_r\subset\CM_r$ are defined in the same way as the 
one-point descendant moduli spaces $\CM^{j_i}_1\subset\CM_1$ by looking at the $r$ tautological line bundles $\LL_{i,r}$ 
over the moduli space $\CM_r = \CM_r(\Gamma^+,\Gamma^-)/\IR$ separately. In other words, we inductively choose generic 
sections $s^j_{i,r}$ in the line bundles $\LL_{i,r}^{\otimes j}$ to define $\CM^{(0,\ldots,0,j,0,\ldots,0)}_r=
(s^j_{i,r})^{-1}(0)\subset\CM^{(0,\ldots,0,j-1,0,\ldots,0)}_r\subset\CM_r$.  \\

With this we can define the descendant Hamiltonian of SFT, which we will continue denoting by $\IH$, while the 
Hamiltonian defined in \cite{EGH} will from now on be called {\it primary}. In order to keep track of the descendants we 
will assign to every chosen differential form $\theta_\alpha$ now a sequence of formal variables $t^{\alpha,j}$ with grading 
\begin{equation*} |t^{\alpha,j}|=2(1-j) -\deg\theta_\alpha. \end{equation*} 
Then the descendant Hamiltonian $\IH\in\hbar^{-1}\WW$ of SFT is defined by 
\begin{equation*}
 \IH = \sum_{\Gamma^+,\Gamma^-,I} \int_{\CM^{(j_1,\ldots,j_r)}_{g,r,A}(\Gamma^+,\Gamma^-)/\IR}
 \ev_1^*\theta_{\alpha_1}\wedge\ldots\wedge\ev_r^*\theta_{\alpha_r}\; \hbar^{g-1}t^Ip^{\Gamma^+}q^{\Gamma^-},
\end{equation*}
where $p^{\Gamma^+}=p_{\gamma^+_1} \ldots p_{\gamma^+_{n^+}}$, $q^{\Gamma^-}=q_{\gamma^-_1} \ldots q_{\gamma^-_{n^-}}$ and 
$t^I=t^{\alpha_1,j_1} \ldots t^{\alpha_r,j_r}$.\\

Expanding 
\begin{equation*} \IH=\hbar^{-1}\sum_g \IH_g \hbar^g \end{equation*} 
we further get a rational Hamiltonian $\Ih=\IH_0\in\PP$, which counts only curves with genus zero. 

\vspace{0.5cm}

\subsection{Quantum Hamiltonian systems with symmetries}
We want to emphasize that the following statement is not yet a theorem in the strict mathematical sense as the analytical 
foundations of symplectic field theory, in particular, the neccessary transversality theorems for the Cauchy-Riemann 
operator, are not yet fully established. Since it can be expected that the polyfold project by Hofer and his 
collaborators sketched in \cite{HWZ} will provide the required transversality theorems, we follow other papers in the field in 
proving everything up to transversality and state it nevertheless as a theorem. 

\begin{theorem} Differentiating the Hamiltonian $\IH\in\hbar^{-1}\WW$ with respect to the formal variables $t_{\alpha,p}$ 
defines a sequence of quantum Hamiltonians
\begin{equation*} \IH_{\alpha,p}=\frac{\del\IH}{\del t^{\alpha,p}} \in H_*(\hbar^{-1}\WW,[\IH,\cdot]) \end{equation*} 
{\it in the full SFT homology algebra with differential $D=[\IH,\cdot]: \hbar^{-1}\WW\to\hbar^{-1}\WW$, which commute with respect to the 
bracket on $H_*(\hbar^{-1}\WW,[\IH,\cdot])$,} 
\begin{equation*} [\IH_{\alpha,p},\IH_{\beta,q}] = 0,\; (\alpha,p),(\beta,q)\in\{1,\ldots,N\}\times\IN. \end{equation*}
\end{theorem}

Everything is an immediate consequence of the master equation $[\IH,\IH]=0$, which can be proven in the same 
way as in the case without descendants using the results in \cite{F2}. While the boundary equation $D\circ D=0$ is well-known 
to follow directly from the identity $[\IH,\IH]=0$, the fact that every $\IH_{\alpha,p}$, $(\alpha,p)\in
\{1,\ldots,N\}\times\IN$ defines an element in the homology $H_*(\hbar^{-1}\WW,[\IH,\cdot])$ follows from the identity
\begin{equation*} [\IH,\IH_{\alpha,p}] = 0,\end{equation*} 
which can be shown by differentiating the master equation with respect to the $t^{\alpha,p}$-variable and using the 
graded Leibniz rule,
\[ \frac{\del}{\del t^{\alpha,p}} [f,g] = 
   [\frac{\del f}{\del t^{\alpha,p}},g] + (-1)^{|t^{\alpha,p}||f|} [f,\frac{\del g}{\del t^{\alpha,p}}]. \]
On the other hand, in order to see that any two $\IH_{\alpha,p}$, 
$\IH_{\beta,q}$ commute {\it after passing to homology} it suffices to see that by differentiating twice (and using that all 
summands in $\IH$ have odd degree) we get the identity
\begin{equation*} 
 [\IH_{\alpha,p},\IH_{\beta,q}]+(-1)^{|t^{\alpha,p}|}[\IH,\frac{\del^2\IH}{\del t^{\alpha,p}\del t^{\beta,q}}] = 0. 
\end{equation*} 

By replacing the full Hamiltonian $\IH$ by the rational Hamiltonian $\Ih$ counting only curves with genus zero, we get a rational version of the above theorem as a corollary.\\

\begin{corollary} Differentiating the rational Hamiltonian $\Ih\in\PP$ with respect to the formal variables $t_{\alpha,p}$ 
defines a sequence of classical Hamiltonians 
\begin{equation*} \Ih_{\alpha,p}=\frac{\del\Ih}{\del t^{\alpha,p}} \in H_*(\PP,\{\Ih,\cdot\}) \end{equation*} 
{\it in the rationall SFT homology algebra with differential $d=\{\Ih,\cdot\}: \PP\to\PP$, which commute with respect to the 
bracket on $H_*(\PP,\{\Ih,\cdot\})$,} 
\begin{equation*} \{\Ih_{\alpha,p},\Ih_{\beta,q}\} = 0,\; (\alpha,p),(\beta,q)\in\{1,\ldots,N\}\times\IN. \end{equation*}
\end{corollary}

We now turn to the question of independence of these nice algebraic structures from the choices like contact form, 
cylindrical almost complex structure, abstract polyfold perturbations and, of course, the choice of the coherent 
collection of sections. This is the content of the following theorem, where we however again want to emphasize that 
the following statement is not yet a theorem in the strict mathematical sense as the analytical foundations of symplectic 
field theory, in particular, the neccessary transversality theorems for the Cauchy-Riemann operator, are not yet fully 
established.   

\begin{theorem} For different choices of contact form $\lambda^{\pm}$, cylindrical almost complex structure 
$\Ju^{\pm}$ , abstract polyfold perturbations and sequences of coherent collections of sections $(s^{\pm}_j)$ the 
resulting systems of commuting operators $\IH^-_{\alpha,p}$ on $H_*(\hbar^{-1}\WW^-,D^-)$ and $\IH^+_{\alpha,p}$ on 
$H_*(\hbar^{-1}\WW^+,D^+)$ are isomorphic, i.e. there exists an isomorphism of the Weyl algebras $H_*(\hbar^{-1}\WW^-,D^-)$ 
and $H_*(\hbar^{-1}\WW^+,D^+)$ which maps $\IH^-_{\alpha,p}\in H_*(\hbar^{-1}\WW^-,D^-)$ to 
$\IH^+_{\alpha,p}\in H_*(\hbar^{-1}\WW^+,D^+)$. 
\end{theorem}

For the proof observe that in \cite{F2} the first author introduced the notion of a collection of sections $(s_j)$ in the 
tautological line bundles over all moduli spaces of holomorphic curves in the cylindrical cobordism interpolating between the
auxiliary structures which are {\it coherently connecting} the two coherent collections of sections $(s^{\pm}_j)$. \\ 

We again get a rational version of the above theorem as a corollary.

\begin{corollary} For different choices of contact form $\lambda^{\pm}$, cylindrical almost complex structure 
$\Ju^{\pm}$ , abstract polyfold perturbations and sequences of coherent collections of sections $(s^{\pm}_j)$ the 
resulting systems of commuting functions $\Ih^-_{\alpha,p}$ on $H_*(\PP^-,d^-)$ and $\Ih^+_{\alpha,p}$ on 
$H_*(\PP^+,d^+)$ are isomorphic, i.e. there exists an isomorphism of the Poisson algebras $H_*(\PP^-,d^-)$ 
and $H_*(\PP^+,d^+)$ which maps $\Ih^-_{\alpha,p}\in H_*(\PP^-,d^-)$ to 
$\Ih^+_{\alpha,p}\in H_*(\PP^+,d^+)$. 
\end{corollary}

These invariance results actually mean that SFT attaches a Weyl (or Poisson) algebra and (quantum) Hamiltonian system therein (in a coordinate free way, i.e. up to algebra isomorphisms) to a stable Hamiltonian structure. Focusing on the rational case for language simplicity we want to point out the fact that \emph{the Poisson SFT homology algebra can be thought of as the space of functions on some abstract infinite-dimensional Poisson space.} The nature of such space can be in fact very exhotic, and the description via space of functions remains the only one with full meaning. \\

However a dynamical interpretation of the Poisson space where the dynamics of the SFT-Hamiltonian system takes place is possible. The kernel $\ker(\{\Ih,\cdot\})$ can be seen as the algebra of functions on the space $\mathcal{O}$ of orbits of the Hamiltonian $\IR$-action given by $\Ih$ (the flow lines of the Hamiltonian vector field $X_{\Ih}$ associated to $\Ih$). Even in a finite dimensional setting the space $\mathcal{O}$ can be very wild. Anyhow the image $\text{im}(\{\Ih,\cdot\})$ is an ideal of such algebra and hence identifies a sub-space of $\mathcal{O}$ given by all of those orbits $o\in \mathcal{O}$ at which, for any $f\in\PP$, $\{\Ih,f\}|_{o}=0$ (notice that, since $\{\Ih,\Ih\}=0$, $\{\Ih,f\}$ descends to a function on $\mathcal{O}$). But such orbits are simply the constant ones, where $X_{\Ih}$ vanishes. \\

Hence the Poisson SFT-homology algebra $H_*(\PP,\{\Ih,\cdot\})$ can regarded as the algebra of functions on $X_{\Ih}^{-1}(0)$, seen as a subspace of the space $\mathcal{O}$ of orbits of $\Ih$, endowed with a Poisson structure by singular, stationary reduction.

\vspace{0.5cm}

\subsection{Integrable hierarchies of Gromov-Witten theory and local SFT}

The above beautiful theory of Hamiltonian systems from Symplectic Field Theory completely contains the older one, coming from rational Gromov-Witten theory (see e.g. \cite{DZ} for a comprehensive treatment), and in some sense even clarifies its topological origin.\\

The special case of stable Hamiltonian structure capturing the full rational Gromov-Witten theory of a closed symplectic manifold $M$ consists in the trivial circle bundle $V=S^1 \times M$. It was noted in \cite{EGH} for all circle bundles over $M$ that the rational SFT-potential is expressed in terms of the rational Gromov-Witten potential of $M$. In the special case of the trivial bundle, the full dispersionless integrable system structure from the Gromov-Witten theory of $M$ is reproduced as the above hamiltonian system of SFT (which is hence integrable, in this case).\\

In particular, further structures than just the Hamiltonian nature are inherited by the SFT-system in the trivial bundle case (and partly in the circle bundle case). In \cite{R2} the second author explicitly described how bihamiltonian and tau-structure are expressed in SFT terms, giving a complete parallelism with Dubrovin's theory of Frobenius manifolds \cite{DZ}. It is a well known fact that one can reconstruct the full descendant genus $0$ Gromov-Witten potential from the primary one by first finding the one-descendant potential, corresponding to the tau-symmetric set of commuting hamiltonians for the relevant integrable system, and then using Witten's conjecture in genus $0$, in its generalized form by Dubrovin, which identifies the full descendant rational potential with the tau-function of a specific solution to such integrable system. This approach has a far reaching extension to higher genera, culminating in Dubrovin's reconstruction method for the full genus $g$ Gromov-Witten potential (\cite{DZ}) in case of semisimple quantum cohomology, while in genus $0$ it is equivalent to string equation, divisor equation, dilaton equation and topological recursion relations.\\

Topological recursion relations can be very naturally expressed in terms of SFT of a circle bundle, thanks to the extra $S^1$-symmetry which is typical of that situation. It is hence very natural to wonder whether topological recursion relations have a counterpart in a more general SFT context than the circle bundle case. While string, dilaton and divisor equations where extended by the authors to the general SFT case in \cite{FR}, topological recursion relations appear far more subtle and we'll see in the next sections how a natural interpretation requires non-$S^1$-equivariant versions of holomorphic curve counting in stable Hamiltonian structures. Since the full picture for a non-equivariant version of SFT is not completely understood yet, we will restrict to the cylindrical case here.\\

Although this is a first concrete step in the study of integrability of the hamiltonian systems of SFT, notice that topological recursion relations in the forms we study here might not be enough to answer the question of integrability: the Hamiltonian systems arising from SFT are, a priori, much more general than those associated with Gromov-Witten invariants, involving in particular more than just local functionals. \\

Here the locality of the Hamiltonians means that the Hamiltonian is given by integrating a function, the so-called Hamiltonian density, over the circle, which itself just reflects the fact that the multiplicities of the positive and negative orbits match. From the latter it follows that the locality condition can be naturally generalized to the case of general contact manifolds by restricting to the Poisson subalgebra $\PP_{\geo}\subset\PP$ of so-called \emph{geometric Hamiltonians}, which is generated by monomials $p^{\Gamma_+}q^{\Gamma_-}$ such that $[\Gamma_+]=[\Gamma_-]\in H_1(V)$. Apart from the fact that it is also easy to see from the definition of the Poisson bracket that $\{f,g\}\in \PP_{\geo}$ whenever $f,g\in\PP_{\geo}$, all Hamiltonians defined by counting holomorphic curves are automatically geometric, i.e., $\Ih\in\PP_{\geo}$ and hence also $\Ih_{\alpha,p}\in\PP_{\geo}$, so that we can only expect to be able to prove completeness of the set of commuting Hamiltonians when we restrict to the above Poisson subalgebra. \\  

While in this paper we still continue to work with the full Poisson algebra $\PP$ as all our results immediately restrict to the above Poisson subalgebra, the  topological recursion relations, even together with string, dilaton and divisor equations, might not yet be restrictive enough to grant complete control over the algebra of commuting Hamiltonians. Nevertheless they seem however to give an affirmative answer to the fundamental question of the reconstructability of the gravitational descendants from the primaries at genus zero.\\

As concrete example beyond the case of circle bundles discussed in \cite{EGH} we review the symplectic field theory of a closed geodesic discussed in \cite{F2}. For this recall that in \cite{F2} the first author has shown that every algebraic invariant of symplectic field theory has a natural analog defined by counting only orbit curves. In particular, in the same way as we define sequences of descendant Hamiltonians $\IH^1_{i,j}$ and $\Ih^1_{i,j}$ by counting general curves in the symplectization of a contact manifold, we can define sequences of descendant Hamiltonians $\IH^1_{\gamma,i,j}$ and $\Ih^1_{\gamma,i,j}$ by just counting branched covers of the orbit cylinder over $\gamma$ with signs (and weights), where the preservation of the contact area under splitting and gluing of curves proves that for every theorem from above we have a version for $\gamma$. \\

Let $\PP^0_{\gamma}$ be the graded Poisson subalgebra of the Poisson algebra $\PP^0$ obtained from the Poisson algebra $\PP$ by setting all $t$-variables to zero, which is generated only by those $p$- and $q$-variables $p_n=p_{\gamma^n}$, $q_n=q_{\gamma^n}$ corresponding to Reeb orbits which are multiple covers of the fixed orbit $\gamma$. In his paper \cite{F2} the first author computed the corresponding Poisson-commuting sequence in the special case where the contact manifold is the unit cotangent bundle $S^*Q$ of a ($m$-dimensional) Riemannian manifold $Q$, so that every closed Reeb orbit $\gamma$ on $V=S^*Q$ corresponds to a closed geodesic $\bar{\gamma}$ on $Q$ (and the string of differential forms just contains a one-form which integrates to one over the closed Reeb orbit). 

\begin{theorem}
The system of Poisson-commuting functions $\Ih^1_{\gamma,j}$, $j\in\IN$ on $\PP^0_{\gamma}$ is isomorphic to a system of Poisson-commuting functions $\Ig^1_{\bar{\gamma},j}$, $j\in\IN$ on $\PP^0_{\bar{\gamma}}=\PP^0_{\gamma}$, where for every $j\in\IN$ the descendant Hamiltonian $\Ig^1_{\bar{\gamma},j}$ given by 
\begin{equation*} 
 \Ig^1_{\bar{\gamma},j} \;=\; \sum \epsilon(\vec{n})\frac{q_{n_1}\cdot ... \cdot q_{n_{j+2}}}{(j+2)!} 
\end{equation*}
where the sum runs over all ordered monomials $q_{n_1}\cdot ... \cdot q_{n_{j+2}}$ with $n_1+...+n_{j+2} = 0$ \textbf{and which are of degree $2(m+j-3)$}. Further $\epsilon(\vec{n})\in\{-1,0,+1\}$ is fixed by a choice of coherent orientations in symplectic field theory and is zero if and only if one of the orbits $\gamma^{n_1},...,\gamma^{n_{j+2}}$ is bad in the sense of \cite{BEHWZ}.
\end{theorem}

For $\gamma=V=S^1$ this recovers the well-known result that the corresponding classical Hamiltonian system with symmetries is given by the dispersionless KdV hierarchy, see \cite{E}. Forgetting about the appearing sign issues, it follows that the sequence $\Ig^1_{\bar{\gamma},j}$ is obtained from the sequence for the circle by removing all summands with the wrong, that is, not maximal degree, so that the system is completely determined by the KdV hierarchy and the Morse indices of the closed geodesic and its iterates. Apart from using the geometric interpretation of gravitational descendants for branched covers of orbit cylinders over a closed Reeb orbit in terms of branching conditions, the second main ingredient for the proof is the idea in \cite{CL} to compute the symplectic field theory of $V=S^*Q$ from the string topology of the underlying Riemannian manifold $Q$ by studying holomorphic curves in the cotangent bundle $T^*Q$. 

\vspace{0.5cm}

\subsection{String, dilaton and divisor equations}
While it is well-known that in Gromov-Witten theory the topological meaning of gravitational descendants leads to new differential 
equations for the Gromov-Witten potential, it is natural to ask how these rich algebraic structures carry over from Gromov-Witten theory to symplectic field theory. As a first step, the authors have shown in the paper \cite{FR} how the well-known string, dilaton and divisor equations generalize from Gromov-Witten theory to symplectic field theory. \\

Here the main problem is to deal with the fact that the SFT Hamiltonian 
itself is not an invariant for the contact manifold. More precisely it depends not only on choices like contact form, 
cylindrical almost complex structure and coherent abstract perturbations but also on the chosen differential forms 
$\theta_i$ and coherent collections of sections $(s_j)$ used to define gravitational descendants. The main application 
of these equations we have in mind is the computation of the sequence of commuting quantum Hamiltonians 
$\IH_{\alpha,p}=\frac{\del\IH}{\del t^{\alpha,p}}$ on SFT homology $H_*(\hbar^{-1}\WW, [\IH,\cdot])$ introduced in the last section. \\

As customary in Gromov-Witten theory we will assume that the chosen string of differential forms on $V$ contains a 
two-form $\theta_2$. Since by adding a marked point we increase the dimension of the moduli space by two, the integration 
of a two-form over it leaves the dimension unchanged and we can expect, as in Gromov-Witten theory, to compute the 
contributions to SFT Hamiltonian involving integration of $\theta_2$ in terms of contributions without integration, where 
the result should just depend on the homology class $A\in H_2(V)$ which can be assigned to the holomorphic curves in the 
corresponding connected component of the moduli space. \\

It turns out that we obtain the same equations as 
in Gromov-Witten theory (up to contributions of constant curves), but these however only hold after passing to SFT homology. 

\begin{theorem}
 For any choice of differential forms and coherent sections the following \emph{string, dilaton and divisor equations} hold \emph{after} passing to SFT-homology 
\begin{eqnarray*}
\frac{\del}{\del t^{0,0}}\IH &=& \int_V t\wedge t  + \sum_{k}t^{\alpha,k+1}\frac{\del}{\del t^{\alpha,k}}\IH 
 \;\in\; H_*(\hbar^{-1}\WW,[\IH,\cdot]), \\
\frac{\del}{\del t^{0,1}}\IH &=& \ID_{\mathrm{Euler}}\IH  \;\in\, H_*(\hbar^{-1}\WW,[\IH,\cdot]), \\
\left(\frac{\del}{\del t^{2,0}} -z_0\frac{\del}{\del z_0}\right)\IH &=& \int_V t\wedge t\wedge \theta_2 + \sum_{k} t^{\alpha,k+1} c_{2\alpha}^\beta\frac{\del\IH}{\del t^{\beta, k}} \;\in\; H_*(\hbar^{-1}\WW,[\IH,\cdot]),
\end{eqnarray*}
with the first-order differential operator 
$$\ID_{\mathrm{Euler}} := -2\hbar\frac{\del}{\del\hbar}-\sum_\gamma p_\gamma\frac{\del}{\del p_\gamma}
-\sum_\gamma q_\gamma\frac{\del}{\del q_\gamma}-\sum_{\alpha,p}t^{\alpha,p}\frac{\del}{\del t^{\alpha,p}}.$$
\end{theorem}

In order to prove the desired equations we will start with special non-generic choices of coherent collections of 
sections in the tautological bundles $\LL_{i,r}$ over all moduli spaces $\CM_r=\CM_{g,r,A}(\Gamma^+,\Gamma^-)/\IR$ and then prove that the resulting equations are covariant with respect to a change of auxiliary data. 

\vspace{0.5cm}

\subsection{Cylindrical contact homology}
While the punctured curves in symplectic field theory may have arbitrary genus and arbitrary numbers of positive and negative punctures, it  is shown in \cite{EGH} that there exist algebraic invariants counting only special types of curves. While in rational symplectic field theory one counts punctured curves with genus zero, contact homology is defined by further restricting to punctured spheres with only one positive puncture. Further restricting to spheres with both just one negative and one positive puncture, i.e. cylinders, the resulting algebraic invariant  is called cylindrical contact homology. \\

Note however that contact homology and cylindrical contact homology are not always defined. In order to 
prove the well-definedness of (cylindrical) contact homology it however suffices to show that there are no punctured holomorphic curves where all punctures are negative (or 
all punctures are positive). To be more precise, for the well-definedness of cylindrical contact homology it actually suffices to assume that there are no holomorphic planes and that there are either no holomorphic cylinders with two positive or no holomorphic cylinders with two negative ends. \\

While the existence of holomorphic curves without positive punctures can be excluded for all contact 
manifolds using the maximum principle, which shows that contact homology is well-defined for all contact manifolds, it can be seen from homological reasons 
that for mapping tori $M_{\phi}$ there cannot exist holomorphic curves in $\IR\times M_{\phi}$ carrying just one type of punctures, which shows that in this case 
both contact homology and cylindrical contact homology are defined. \\ 

Similarly to what happens in Floer homology, the chain space for cylindrical homology $C$ is defined be the vector space generated by the formal variables $q_{\gamma}$ with coefficients which are formal power series in the $t^{\alpha,j}$-variables and Laurent series in the $z_n$-variables. Counting only holomorphic cylinders defines  a differential $\del: C_*\to C_*$ by 
\[ \del q_{\gamma^+} = \sum_{\gamma^-} \frac{\del^2 \Ih}{\del p_{\gamma^+}\del q_{\gamma^-}}|_{p=q=0} \cdot q_{\gamma^-}\] 
with $\del\circ\del =0$ when there do not exist any holomorphic planes, so that one can define the cylindrical homology of the closed stable Hamiltonian manifold as the homology of the chain complex $(C,\del)$. The sequence of commuting Hamiltonians $\Ih_{\alpha,p}$ in rational symplectic field theory gets now replaced by linear maps 
\[\del_{\alpha,p} = \frac{\del}{\del t^{\alpha,p}}\circ \del: C_*\to C_*,\; 
\del_{\alpha,p}q_{\gamma^+}= \sum_{\gamma^-} \frac{\del^3 \Ih}{\del t^{\alpha,p}\del p_{\gamma^+}\del q_{\gamma^-}}|_{p=q=0} \cdot q_{\gamma^-},\]
which by the same arguments descend to maps on homology, $\del_{\alpha,p}: H_*(C,\del)\to H_*(C,\del)$, and commute on homology, $[\del_{\alpha,p},\del_{\beta,q}]_-=0$, with respect to the graded commutator $[f,g]_-=f\circ g - (-1)^{\deg(f)\deg(g)} g\circ f$.\\

While we have already shown how the well-known string, dilaton and divisor equations translate from Gromov-Witten theory to SFT, in this paper we want to proceed with our project of understanding how the rich algebraic structures from Gromov-Witten theory carry over to symplectic field theory.  As the next step we want to show how classical genus-zero topological recursion generalizes to symplectic field theory. As we will outline in a forthcoming paper, it follows from the computation of the SFT of a Reeb orbit with descendants outlined above that the genus-zero topological recursion requires a non-equivariant version of SFT, which is generated by parametrized instead of unparametrized Reeb orbits. \\

The definition of this non-equivariant version of SFT is currently a very active field of research and related to the work of Bourgeois and Oancea in \cite{BO}, where a Gysin-type spectral sequence relating linearized contact homology (a slight generalization of cylindrical contact homology depending on a symplectic filling) and symplectic homology of this filling is established by viewing the one as the (non-)equivariant version of the other. On the other hand, since the topological recursion relations are already interesting in the case of cylindrical contact homology and the non-equivariant version of it is already understood, in this paper we will study the algebraic structure of gravitational descendants just for this special case first. 

\vspace{0.5cm}

\section{Topological recursion in Floer homology from Gromov-Witten theory}
It is well-known that there is a very close relation between Gromov-Witten theory and symplectic Floer theory. Apart from the fact that the Floer (co)homology groups are isomorphic to quantum cohomology groups of the underlying symplectic manifold, which was used by Floer to prove the famous Arnold conjecture about closed orbits of the Hamiltonian vector field, it was outlined by Piunikhin-Salamon-Schwarz in \cite{PSS} how the higher structures carry over from one side to the other.  As a motivation for the topological recursion relations in (non-equivariant) cylindrical homology, which we will prove in the next section, we show in this section how the topological recursion relations in Gromov-Witten theory carry over to symplectic Floer homology using a Morse-Bott correspondence.

\vspace{0.5cm}

\subsection{Topological recursion in Gromov-Witten theory}
As already mentioned in the last section it is customary in Gromov-Witten theory to introduce $r$ tautological line 
bundles $\LL_1,\ldots,\LL_r$ over each moduli space $\CM_r=\CM_{g,r,A}(X)$ of closed $J$-holomorphic curves $u:(S_g,j)\to(X,J)$, where the fibre of $\LL_i$ over a curve $(u,\Si)\in\CM_r$ is again given 
by the cotangent line to the underlying, possibly unstable nodal Riemann surface (without ghost components) at the 
$i$.th marked point and which again formally can be defined as the pull-back of the vertical cotangent line 
bundle of $\pi: \CM_{r+1}\to\CM_r$ under the canonical section $\sigma_i: \CM_r\to\CM_{r+1}$ mapping to the $i$.th marked 
point in the fibre. \\

Assuming we have chosen a basis $A_0,\ldots,A_N$ of $H_2(X)$, we assign to every $A_i$ a formal 
variable $z_i$ with grading $|z_i|=- 2 c_1(A_i)$. In order to include higher-dimensional moduli spaces we further assume that a string of closed (homogeneous) differential forms $\Theta=(\theta_1,\ldots,\theta_N)$ on $X$ is chosen and assign to 
every $\theta_{\alpha}\in\Omega^*(X)$ a sequence of formal variables $t^{\alpha,j}$ 
with grading
\begin{equation*} |t^{\alpha,j}|=2 - 2j - \deg\theta_{\alpha}. \end{equation*}  
Finally, let $\hbar$ be another formal variable of degree $|\hbar|=2(m-3)$ with $\dim X=2m$. \\

With this we can define the descendant potential of Gromov-Witten theory, which we will denote by $\IF$,
\begin{equation*}
 \IF = \sum_{I} \int_{\CM^{(j_1,\ldots,j_r)}_{g,r,A}(X)}
 \ev_1^*\theta_{\alpha_1}\wedge\ldots\wedge\ev_r^*\theta_{\alpha_r}\; \hbar^{g-1}t^I z^d,
\end{equation*}
where $t^{I}=t^{\alpha_1,j_1} \ldots t^{\alpha_r,j_r}$ and $z^d = z_0^{d_0} \cdot \ldots \cdot z_N^{d_N}$. 
Expanding 
\begin{equation*} \IF=\hbar^{-1}\sum_g \IF_g \hbar^g \end{equation*} 
we further get a rational potential $\If=\IF_0\in\PP$, which counts only curves with genus zero. \\

Topological recursion relations are differential equations for the descendant potential which are proven using the geometric meaning of gravitational descendants, as for the string, dilaton and divisor equations. They follow from the fact that by making special non-generic choices for the sections in the tautological line bundles, the resulting zero divisors $\CM_{0,r}^{(j_1,\ldots,j_r)}(X)$ localize on the singular strata in the compactified moduli space $\CM_r(X)$. \\ 

We start with the following definition. A $r$-labelled tree is a triple $(T,E,\Lambda)$, where $(T,E)$ is a tree with the set of 
vertices $T$ and the edge relation $E \subset T\times T$. The set $\Lambda = (\Lambda_{\alpha})$ is a decomposition of the index 
set $I=\{1,\ldots,r\}=\bigcup \Lambda_{\alpha}$. We write $\alpha E\beta$ if $(\alpha,\beta)\in E$. A nodal holomorphic curve of genus zero modelled over $T=(T,E,\Lambda)$ is a tuple $(\underline{u}=(u_{\alpha}),\uz = 
((z_{\alpha\beta})_{\alpha E\beta}, (z_k))$ of holomorphic maps $u_{\alpha}: (\CP,i)\to (X,J)$ and special points $z_{\alpha\beta}, z_k \in \CP$ such that $u_{\alpha}(z_{\alpha\beta})=u_{\beta}(z_{\beta\alpha})$ and for each $\alpha\in T$ the special points 
in $Z_{\alpha} = \{z_{\alpha\beta}:\alpha E\beta\}\cup\{z_k: k\in\Lambda_{\alpha}\}$ are pairwise distinct. A nodal holomorphic curve $(\underline{u},\uz)$ is called stable if every constant map $u_{\alpha}$ the underlying sphere carries at 
least three special points. Denote by $\IM_T=\IM_T(X)$ the space of all nodal holomorphic curves (of genus zero) modelled 
over the tree $T=(T,E,\Lambda)$. Taking the union of all moduli spaces of stable nodal holomorphic curves modelled over $r$-labelled trees, we obtain the compactified space, $\CM_r = \coprod_T \IM_T$, which, equipped with the Gromov topology, provides the compactification of the moduli space $\IM_r$ of holomorphic spheres with $r$ marked points. \\
 
Fix $1\leq i\leq r$ and choose $1\leq j,k\leq r$ such that $i,j,k$ are pairwise different. While the string, dilaton and divisor equations are proven (see also \cite{FR}) by studying the behavior of the tautological line bundle under the natural map $\CM_{0,r}(X)\to\CM_{0,r-1}(X)$, the topological recursion relations follow by studying the behavior of the tautological line bundle under the natural map $\CM_{0,r}(X)\to\CM_{0,3}=\{\textrm{point}\}$, where the map and all marked points except $i,j,k$ are forgotten. Defining the divisor $\CM_{0,r}^{i,(j,k)}(X)$ as the union of moduli spaces of nodal holomorphic curves
\[\CM_{0,r}^{i,(j,k)}(X) = \bigcup_{T_0}\CM_{T_0}(X) = \bigcup_{T\leq T_0} \IM_T(X),\] 
over all labelled trees $T_0=(T_0,E_0,\Lambda_0)$ with $T_0=\{1,2\}$, $1E_02$, $i\in\Lambda_{0,1}$, $j,k\in\Lambda_{0,2}$ or $T=(T,E,\Lambda)$ with $i\in\Lambda_{\alpha}\Rightarrow j,k\not\in \Lambda_{\alpha}$, respectively, the following theorem is a standard result in Gromov-Witten theory. \\

\begin{theorem}  By studying the behavior of the tautological line bundle under the map $\CM_{0,r}(X)\to\CM_{0,3} =\{\textrm{point}\}$, where the map and all marked points except $i,j,k$ are forgotten, one can construct a special non-generic section in the tautological line bundle $\LL_{i,r}$ over $\CM_{0,r}(X)$ such that the zero divisor $\CM_{0,r}^{(0,\ldots,0,1,0,\ldots,0)}(X)$ agrees with the above divisor $\CM_{0,r}^{i,(j,k)}(X)$ of holomorphic spheres with one node, where the $i$.th marked point lies on one component and the $j$.th and the $k$.th fixed marked points lie on the other component. \end{theorem}

On the other hand, translating this into a differential equation for the descendant potential $\If$ we get
\begin{corollary} The descendant potential $\If$ of rational Gromov-Witten theory satiesfies the following topological recursion relations.
\[ \frac{\del^3 \If}{\del t^{\alpha,i}\del t^{\beta,j}\del t^{\gamma,k}} = \frac{\del^2 \If}{\del t^{\alpha,i-1}\del t^{\mu,0}} \,\eta^{\mu\nu} 
\frac{\del^3 \If}{\del t^{\nu,0}\del t^{\beta,j}\del t^{\gamma,k}}, \]
where $\eta$ is the metric given by Poincar\'e pairing.
\end{corollary}

Before we show how these topological recursion relations translate from Gromov-Witten theory to symplectic Floer homology, we want to introduce a slightly weaker version of the above relations. Note that for the above formula one needs to choose two auxiliary $t$-variables. It follows that the recursion is not symmetric with respect to permuting of marked points. On the other hand, it will turn out that in symplectic Floer homology, more generally, non-equivariant cylindrical contact homology, we need to treat all additional marked points in a symmetric way to obtain coherence for the chosen special collections of sections. Because of this fact, we emphasize that the above topological recursion relations immediately lead to the following averaged version. 
\begin{corollary} The descendant potential $\If$ of rational Gromov-Witten theory satiesfies the following averaged version of topological recursion relations,
\[ N(N-1) \frac{\del\If}{\del t^{\alpha,i}} = \frac{\del^2 \If}{\del t^{\alpha,i-1}\del t^{\mu,0}} \,\eta^{\mu\nu} 
N(N-1)\frac{\del\If}{\del t^{\nu,0}}, \]
where $N:=\displaystyle{\sum_{\beta,j}}t^{\beta,j}\frac{\del}{\del t^{\beta,j}}$ is the differential operator which counts the number of marked points.
\end{corollary}
For the proof observe that the commutator of $\displaystyle{\frac{\del}{\del t^{\beta_1,j_1}}}$ and $t^{\beta_2,j_2}$ is one when $(\beta_1,j_1)=(\beta_2,j_2)$ and vanishes in all other cases, so that  $$N(N-1)=\sum_{\beta_1,j_1}\sum_{\beta_2,j_2}t^{\beta_1,j_1}t^{\beta_2,j_2}\frac{\del}{\del t^{\beta_1,j_1}}\frac{\del}{\del t^{\beta_2,j_2}}.$$

\vspace{0.5cm}

\subsection{From Gromov-Witten theory to Floer homology}
Recall that the goal of this section is to translate the above topological recursion relations from Gromov-Witten theory to symplectic Floer theory. In order to do so, we recall in this subsection the well-known relation between Gromov-Witten theory and symplectic Floer theory. \\

First recall that the Floer cohomology $HF^*=HF^*(H)$ of a time-dependent Hamiltonian $H: S^1\times X\to \IR$ is defined as the homology of the chain complex $(CF^*,\del)$, where $CF^*$ is the vector space freely generated by the formal variables $q_{\gamma}$ assigned to all one-periodic orbits of $H$ with coefficients which are Laurent series in the variables $z_n$. On the other hand, the differential $\del: CF^*\to CF^*$ is given by counting elements in the moduli spaces $\CM(\gamma^+,\gamma^-)$ of cylinders $u:\IR\times S^1\to X$ satisfying the perturbed Cauchy-Riemann equation $\CR_{J,H}(u)=\del_s u+J_t(u)(\del_t u - X^H_t(u)) =0$ with a one-periodic family of almost complex structures $J_t$ and where $X^H_t$ denotes the symplectic gradient of $H_t$, and which converge to the one-periodic orbits $\gamma^{\pm}$ as $s\to\pm\infty$. \\

In the same way as the group of Moebius transforms acts on the solution space of Gromov-Witten theory and the moduli space is defined only after dividing out this obvious symmetries,  $\IR$ acts on the above space of Floer cylinders by translations in the domain, so that the moduli space is again defined after dividing out this natural action. On the other hand, since the Hamiltonian (and the almost complex structure) depends on the $S^1$-coordinate, it will become important that we do not divide out the action of the circle. \\

In order to prove the Arnold conjecture about the number of one-periodic orbits of $H$ one shows that the Floer cohomology groups are isomorphic to the quantum cohomology groups $QH^*(X)$ of the underlying symplectic manifold, which are defined as the vector space freely generated by formal variables $t^{\alpha}=t^{\alpha,0}$, assigned as before to a chosen basis of the singular cohomology of $X$, with coefficients which are Laurent series in the $z_n$-variables. Note that as vector spaces the only difference to the usual cohomology groups $H^*(X)$ lies in the different choice of coefficients. \\
 
One way to prove the above isomorphism is by studying the behavior of the moduli spaces of Floer cylinders as the Hamiltonian $H$ converges to zero. Since in the limit $H=0$ every point on $X$ corresponds to a closed orbit, i.e. the orbits are no longer isolated but appear in finite-dimensional manifolds, we arrive at a Morse-Bott version of Floer cohomology in the sense of Bourgeois, see \cite{B}. \\

It follows from the removable singularity theorem for (unperturbed) holomorphic curves that in the limit the moduli spaces of Floer trajectories $\CM(\gamma^+,\gamma^-)$ are replaced by the moduli spaces of holomorphic spheres $\CM_{0,2}(X)$ with two marked points from Gromov-Witten theory, where at each of the two points cohomology classes $\alpha^+, \alpha^-\in H^*(X)$ from the target manifold $X$ are pulled-back using the natural evaluation map. It follows that in the limit $H=0$ the chain space is given by the vector space freely generated by formal variables $t^{\alpha}$ with coefficients which are Laurent series in the $z_n$-variables.\\

While holomorphic spheres with two marked points contribute to the Gromov-Witten potential of $X$, it is very important to observe that in the Morse-Bott limit $H=0$ we  no longer divide out all symmetries $\IR\times S^1$ of the target, but only the $\IR$-shift as in Floer homology. It follows that the moduli space of (non-trivial) holomorphic spheres in the Morse-Bott limit always carries a non-trivial $S^1$-symmetry. In particular, it follows that differential for $H=0$ is indeed zero, so that the quantum cohomology groups agree with the underlying chain groups, which proves that the Floer cohomology groups are indeed isomorphic to the quantum cohomology groups, $HF^*(H)\cong QH^*(X)$.  \\

Next we want to compare the natural operations on Floer and quantum cohomology. First note that there is a natural product structure on quantum cohomology, the so-called quantum cup product, given by counting holomorphic spheres with three marked points,
\[ QH^*(X)\otimes QH^*(X)\to QH^*(X),\; (t^{\alpha},t^{\beta}) \mapsto \frac{\del^3 \If}{\del t^{\alpha} \del t^{\beta} \del t^{\mu}} \eta^{\mu\nu} t^{\nu},\]
which we can also view as an action of $QH^*(X)$ on itself. Keeping the position of the first two marked points still fixed, note that we assume that also the position of the third marked point is fixed to determine unique coordinates on $\CP$.  \\

On the other hand, in \cite{PSS} it was already shown that one can define the corresponding action of $QH^*(X)$ on the Floer cohomology groups $HF^*(H)$ by counting Floer cylinders $u:\IR\times S^1\to X$, $\CR_{J,H}(u)=0$ with an additional marked point with fixed position $(0,0)\in\RS$, as in the description of the moduli spaces of the quantum product. By the same arguments as used for the differential, note that when we pass to the Morse-Bott limit $H=0$ it is clear that this just gives us back the quantum cup product defined above. \\

Note that on the quantum cohomology side we can either assume that also the third marked point is fixed or it varies and we divide out the symmetry group $\RS$ of the two-punctured sphere afterwards. On the other hand, since in the Floer case we now only divide by $\IR$ and not by $\RS$ on the domain, it follows that in the second case with varying position the third marked point must still be constrained to lie on some ray $\IR\times\{t_0\} \subset \RS$, where without loss of generality we can assume that $t_0=0$. Keeping the picture of points with varying positions of marked points as in Gromov-Witten theory and SFT we hence have the following 

\begin{proposition} With respect to the above Morse-Bott correspondence, counting holomorphic spheres with three marked points in Gromov-Witten theory corresponds in symplectic Floer theory to counting Floer cylinders with one additional marked point constrained to $\IR\times\{0\}\subset \RS$. \end{proposition}

More generally, adding an arbitrary number of additional marked points to include the full Gromov-Witten potential, it follows that we indeed have the following 

\begin{proposition}  With respect to the above Morse-Bott correspondence, counting holomorphic spheres with three or more marked points in Gromov-Witten theory corresponds on the Floer side to counting Floer cylinders with additional marked points, where only the first marked point is constrained to $\IR\times\{0\}\subset \RS$.  \end{proposition}

\vspace{0.5cm}

\subsection{Topological recursion in Floer homology}
With the above translation scheme from Gromov-Witten theory to symplectic Floer theory in hand, we now will transfer the topological recursion relations from Gromov-Witten theory, outlined in the first subsection, to symplectic Floer theory. \\

But before we can give a precise algebraic statement, the propositions of the last subsection suggest that we first need to enrich the algebraic formalism as follows. As in Gromov-Witten theory and SFT we now enrich the Floer complex by requiring that the underlying chain space $CF^*$ is again the vector space freely generated by the formal variables $q_{\gamma}$ but with coefficients which are now not only Laurent series in the $z_n$-variables, but also formal power series in the $t^{\alpha,j}$-variables. \\

Denote by $\CM_r(\gamma^+,\gamma^-)$ the moduli space of Floer trajectories with $r$ marked points. Introducing $r$ tautological line bundles $\LL_{i,r}$ on $\CM_r(\gamma^+,\gamma^-)$ and coherent collections of sections as in cylindrical contact homology (and without distinguishing between constrained and unconstrained points), we again denote by $\CM_r^{(j_1,\ldots,j_r)}(\gamma^+,\gamma^-)$ the corresponding zero divisors. \\

As in cylindrical contact homology we can then enrich the differential $\del: CF^*\to CF^*$ by pulling back differential forms from the target $X$ and introducing descendants, 
\[\del(q_{\gamma^+}) = \sum \int_{\CM^{(j_1,\ldots,j_r)}_r(\gamma^+,\gamma^-)} \ev_1^*\theta_{\alpha_1}\wedge \ldots \wedge \ev_r^*\theta_{\alpha_r} \; t^I  q_{\gamma^-} z^d\]
with $t^I=t^{\alpha_1,j_1} \ldots t^{\alpha_r,j_r}$. As in cylindrical contact homology we then define 
\[ \del_{(\alpha,i)} := \frac{\del}{\del t^{\alpha,i}}\circ \del.\]

As we have seen in the last subsection we further need to include cylinders with one constrained (to $\IR\times\{0\}$) marked point. In order to distinguish these new linear maps from the linear maps $\del_{\alpha,p}$ obtained by counting holomorphic cylinders with one \emph{un}constrained marked point, we denote them by $\del_{\check{\alpha},p}: CF^*\to CF^*$. In the same as for $\del_{\alpha,p}$  it can be shown that $\del_{\check{\alpha},p}$ descends to a linear map on homology, and commutes on homology, $[\del_{\check{\alpha},p},\del_{\check{\beta},q}]_-=0$, with respect to the graded commutator $[f,g]_-=f\circ g - (-1)^{\deg(f)\deg(g)} g\circ f$.\\

With this we can formulate our proposition about topological recursion in symplectic Floer theory as follows. In contrast to Gromov-Witten theory, we now obtain three different equations, depending on whether we remember both punctures, one puncture and one additional marked point or two marked points. On the other hand, inside their class we want to treat the additional marked point in a symmetric way as in the averaged version of topological recursion relations in Gromov-Witten theory stated in subsection 3.1. Note that this is needed for coherence as we will show in the proof of the corresponding relations for non-equivariant cylindrical contact homology in subsection 4.3.

\begin{proposition}\label{TRR-Floer} With respect to the above Morse-Bott correspondence, the topological recursion relations from Gromov-Witten theory have the following translation to symplectic Floer theory: For three \emph{different} non-generic special choices of coherent sections we have 
\begin{itemize}
\item[(2,0):] $$ \del_{(\check{\alpha},i)} = \frac{\del^2 \If}{\del t^{\alpha,i-1}\del t^{\mu}} \eta^{\mu\nu} \del_{\check{\nu}}$$
\item[(1,1):] $$ N\,\del_{(\check{\alpha},i)} = \frac{\del^2 \If}{\del t^{\alpha,i-1}\del t^{\mu}} \eta^{\mu\nu} N\, \del_{\check{\nu}} + \frac{1}{2}[\del_{(\check{\alpha},i-1)}, \check{N}\,\del]_+ $$
\item[(0,2):] $$N(N-1)\,\del_{(\check{\alpha},i)} = \frac{\del^2 \If}{\del t^{\alpha,i-1}\del t^{\mu}} \eta^{\mu\nu} N(N-1)\, \del_{\check{\nu}} +[\del_{(\check{\alpha},i-1)},\check{N}(N-1)\,\del]_+$$
\end{itemize}
\vspace{0.5cm}
where $N:=\displaystyle{\sum_{\beta,j}}t^{\beta,j}\frac{\del}{\del t^{\beta,j}}$, $\check{N}:=\displaystyle{\sum_{\beta,j}} t^{\beta,j}\frac{\del}{\del \check{t}^{\beta,j}}$ and $[f,g]_+=f\circ g +(-1)^{\deg(f)deg(g)} g\circ f$ denotes a graded anti-commutator with respect to the operator composition. Notice that, $\check{N}$ and $\del$ having odd degree, in the above formulas the anti-commutator always corresponds to a sum.
\end{proposition}
\begin{proof}
In order to translate the localization result of Gromov-Witten theory to symplectic Floer theory, we first replace the holomorphic sphere with three or more marked points by a Floer cylinder with one marked point constrained to $\IR\times\{0\}$ and possibly other unconstrained additional marked points, where we assume that the constrained marked point agrees with the $i$.th marked point carrying the descendant. In order to obtain the three different equations we have to decide whether the $j$.th and $k$.th marked point agree with the positive or negative puncture or some other additional marked point and then use the localization theorem from the first subsection, which states that the zero divisor localizes on nodal spheres with two smooth components, where the $i$.th marked point lies on one component and the $j$.th and the $k$.th marked point lie on the other component. \\

In order to obtain equation (2,0) we remember both punctures ($j=+, k=-$). While for each holomorphic curve in the corresponding divisor $\CM^{i,(+,-)}_r(\gamma^+,\gamma^-)\subset \CM_r(\gamma^+,\gamma^-)$ the component with the $j$.th and the $k$.th marked point is a Floer cylinder, the other component with the $i$.th marked point is a sphere, since both components are still connected by a node (and not a puncture) by the compactness theorem in Floer theory. \\

On the other hand, in order to obtain equation (1,1), we remember one of the two punctures, $j=+$ (or $j=-$) and another marked point. While for each holomorphic curve in $\CM^{i,(+,k)}_r(\gamma^+,\gamma^-)\subset\CM_r(\gamma^+,\gamma^-)$ the component carrying the $j$.th and the $k$.th marked point still needs to be a Floer cylinder and not a holomorphic sphere as the $j$.th marked point is a puncture, the other component carrying the $i$.th marked point can either be a holomorphic sphere or Floer cylinder, depending on whether both components are connected by a node or a puncture. Note that in the second case this connecting puncture is neccessarily the negative puncture for the Floer cylinder with the $i$.th marked point and the positive puncture for the Floer cylinder with the $k$.th marked point. On the other hand, both Floer cylinder carry a special marked point, namely the $i$.th or the $k$.th marked point, respectively, which by the above Morse-Bott correspondence are constrained to $\IR\times\{0\}$. \\

Finally, in order to establish equation (0,2), we remember none of the two punctures. Since only Floer cylinders and holomorphic spheres appear in the compactification, it follows that for the above equation we must just sum over all choices for both components being either a cylinder or a sphere and again use the above Morse-Bott correspondence.
\end{proof} 

\begin{remark} Note that in order to make the above proof precise, one needs to rigorously establish an isomorphism between Gromov-Witten theory and symplectic Floer theory \emph{beyond} the isomorphism between quantum cohomology and Floer cohomology groups together with the action of the quantum cohomology on them proven in \cite{PSS} (involving the full 'infinity structures'). On the other hand, while we expect that the Morse-Bott picture from above should lead to such an isomorphism in an obvious way, we are satisfied with the level of rigor (and call our result 'proposition' instead of 'theorem'), not only since it should just serve as a motivation for our topological recursion result in non-equivariant cylindrical contact homology, but also since our rigorous proof for that case will in turn directly lead to a (rigorous) proof of this proposition. \end{remark}

\vspace{0.5cm}
 
\section{Topological recursion in non-equivariant cylindrical homology}

Motivated by the topological recursion result in symplectic Floer homology discussed in the last section, we want to prove the corresponding topological recursion result for cylindrical contact homology. Recall that for the well-definedness of cylindrical contact homology we will assume that there are no holomorphic planes and that there are either no holomorphic cylinders with two positive or no holomorphic cylinders with two negative ends. Recall further that this is satisfied in the contact case when there are no holomorphic planes or when the stable Hamiltonian manifold is a symplectic mapping torus. Since, in contrast to Floer homology, the closed orbits in cylindrical contact homology are not parametrized by $S^1$, it turns out that we need to work with a non-$S^1$-equivariant version of cylindrical contact homology, where we follow the ideas in Bourgeois-Oancea's paper \cite{BO}. Note that in the contact case our results indeed immediately generalize from cylindrical contact homology to linearized contact homology which depends on a symplectic filling and is defined for any fillable contact manifold, e.g. since it is still isomorphic to the positive symplectic homology which only counts true cylinders.

\vspace{0.5cm}

\subsection{Non-equivariant cylindrical homology}

Recall that in the last section about topological recursion in Gromov-Witten theory and symplectic Floer theory we assumed that the underlying symplectic manifold $X$ is closed, i.e. compact and without boundary. On the other hand, it is well-known that one can also define a version of Floer homology for compact manifolds $X$ with contact boundary $V=\del X$, which is called symplectic homology. Note that here and in what follows we make no distinction between $X$ and its completion $X\cup\IR^+\times \del X$. It is defined, see also \cite{BO} for details, as a direct limit of Floer homology groups $SH_*(X)=\overrightarrow{\lim}\; HF_*(H)$ over an ordered set of admissible time-dependent Hamiltonians $H: S^1\times X\to\IR$. \\

On the other hand, it follows from the definition of the set of admissible Hamiltonians that, in the case when the Hamiltonian is time-independent, the one-periodic orbits of $H$ are critical points of $H$ in the interior of $X$ or correspond to closed orbits of the Reeb vector field on the contact boundary $V=\del X$, where both sets of closed orbits can be distinguished by the symplectic action. While it is easily seen that the chain subcomplex generated by the critical points computes the singular cohomology of $X$, $H^*(X)\cong H_{\dim X-*}(X,\del X)$, it follows using the action filtration that one can define $SH^+_*(X)$, the positive part of symplectic homology, or positive symplectic homology in short, as the homology of the corresponding quotient complex.  \\

While for time-independent Hamiltonians the generators of positive symplectic homology correspond to closed orbits of the Reeb vector field on $V=\del X$, note that in the definition of (positive) symplectic homology, in the same way as in Floer homology, it is explicitly assumed that the underlying Hamiltonians $H: S^1\times X\to \IR$ are $S^1$-dependent. In particular, it follows that, in contrast to the cylindrical homology $HC_*(V)$, the positive symplectic homology $SH^+_*(X)$ is generated by parametrized instead of unparametrized orbits. \\
  
In their paper \cite{BO} the authors have proven that positive symplectic homology and cylindrical contact homology are related via a Gysin-type sequence. To be more precise, in their paper they use linearized contact homology instead of cylindrical contact homology, since the first one is defined for all contact manifolds with a strong symplectic filling, and agrees with cylindrical contact homology as long as there no holomorphic planes in the filling. For the proof Bourgeois and Oancea have defined a non-equivariant version of linearized contact homology, which is generated by parametrized instead of unparametrized orbits, and shown that it is isomorphic to positive symplectic homology. \\

The definition of this non-equivariant version of cylindrical contact homology is motivated by the following Morse-Bott approach to positive symplectic homology starting from an admissible set of time-\emph{in}dependent Hamiltonians $H^0: X\to\IR$. Recall that in this case the set of closed parametrized orbits of the Hamiltonian vector field is given by the set of closed parametrized Reeb orbits, which in turn is obtained from the usual set of set of closed (unparametrized) Reeb orbits by assigning to every closed unparametrized orbit the corresponding $S^1$-family of closed parametrized orbits. While in this case the closed Hamiltonian orbits are no longer nondegenerate, in particular, no longer isolated as required in the usual definition of symplectic (Floer) homology,  but still come in $S^1$-families, the authors of \cite{BO} resolved the problem by applying a Morse-Bott approach as in Bourgeois' thesis \cite{B}.\\

Indeed, in order to obtain from a time-independent Hamiltonian $H^0: X\to\IR$ a family of time-dependent Hamiltonians $H^{\epsilon}: S^1\times X\to\IR$ with only nondegenerate closed orbits, one can choose for each closed (unparametrized) orbit $\gamma$ of the Hamiltonian $H^0$ a self-indexing Morse function $f_{\gamma}: \gamma \cong S^1\to\IR$, i.e. with one maximum and one minimum. Further choosing for each closed orbit a tubular neigborhood $(r,t)\in (-1,+1)\times S^1 \hookrightarrow X$ such that $\{0\}\times S^1$ corresponds to $\gamma$, and a smooth cut-off function $\varphi: (-1,+1)\to\IR$, so that we can extend each chosen Morse function $f_{\gamma}: S^1\to\IR$ to a function on $X$ by $\tilde{f}_{\gamma}(r,t) = \varphi(r)\cdot f_{\gamma}(t)$ and zero-extension, we can define the family of perturbed time-dependent Hamiltonians $H^{\epsilon}: S^1\times X\to\IR$ by $H^{\epsilon}(t,\cdot) = H^0+\epsilon\cdot \sum_{\gamma} \tilde{f}_{\gamma}(\cdot,t)$.  \\

It can be shown that the closed (parametrized) orbits of each Hamiltonian $H^{\epsilon}$ are in one-to-one correspondence with the critical points of the Morse functions $f_{\gamma}$. Since all the Morse functions $f_{\gamma}: S^1\to\IR$ were chosen to have precisely one maximum $\hat{\gamma}$ and one minimum $\check{\gamma}$, it follows that each closed unparametrized orbit of the time-independent Hamiltonian $H^0:X\to\IR$ gives rise to \emph{two} closed parametrized orbits $\hat{\gamma}_{\epsilon}$ and $\check{\gamma}_{\epsilon}$ for each time-dependent Hamiltonian $H^{\epsilon}: S^1\times X\to\IR$. \\

Similarly to the approach in Bourgeois' thesis \cite{B}, in \cite{BO} the authors define a Morse-Bott complex for positive symplectic homology $SH^+_*(X)$ based on admissible time-independent Hamiltonians $H^0$ whose generators are the critical points $\hat{\gamma}$ and $\check{\gamma}$ on all closed unparametrized orbits $\gamma$ of $H^0$ and whose differential counts 'cascades' consisting of solutions to the Floer equation for $H^0$ connected by gradient trajectories of the Morse functions $f_{\gamma}$. On the other hand, it is clear that this immediately generalizes to cylindrical contact homology in the sense that one can define a non-equivariant version of cylindrical contact homology by choosing self-indexing Morse functions $f_{\gamma}$ for all closed (unparametrized) Reeb orbits $\gamma$. \\

Indeed, with this we can define non-equivariant cylindrical contact homology as the homology of the Morse-Bott complex whose generators are the critical points $\hat{\gamma}$ and $\check{\gamma}$ on all closed Reeb orbits $\gamma$ and whose differential counts 'cascades' consisting of holomorphic cylinders in $\IR\times V$ connected by gradient trajectories of the Morse functions $f_{\gamma}$. In other words, introducing two formal variables $\hat{q}_{\gamma}$ and $\check{q}_{\gamma}$ for each closed Reeb orbit $\gamma$ (corresponding to the two critical points $\hat{\gamma}$ and $\check{\gamma}$), the chain space for non-equivariant cylindrical contact homology $HC^{\textrm{non-}S^1}_*(V)$ is the vector space generated by the formal variables $\hat{q}_{\gamma}$ and $\check{q}_{\gamma}$ with coefficients which are formal power series in the $t^{\alpha,j}$-variables and Laurent series in the $z_n$-variables. Note that the chain space naturally splits, $$C^{\textrm{non-}S^1}_*=\hat{C}_*\oplus\check{C}_*,$$ where $\hat{C}_*$, $\check{C}_*$ is generated by the formal variables $\hat{q}_{\gamma}$, $\check{q}_{\gamma}$, respectively. Defining the differential $\del: \hat{C}_*\oplus\check{C}_*\to\hat{C}_*\oplus\check{C}_*$ as described above by counting Morse-Bott cascades, we furthermore 
define as in Floer homology, \[ \del_{(\alpha,i)} := \frac{\del}{\del t^{\alpha,i}}\circ \del: \hat{C}_*\oplus\check{C}_*\to\hat{C}_*\oplus\check{C}_*.\]
Furthermore as for Floer homology we further need to include cylinders with one constrained (to $\IR\times\{0\}$) marked point. In order to distinguish these new linear maps from the linear maps $\del_{\alpha,p}$ obtained by counting holomorphic cylinders with one \emph{un}constrained marked point, we again denote them by $\del_{\check{\alpha},p}: CF^*\to CF^*$. In the same way as for $\del_{\alpha,p}$,  it can be shown that $\del_{\check{\alpha},p}$ descends to a linear map on non-equivariant cylindrical homology.

\vspace{0.5cm}

\subsection{Topological recursion in non-equivariant cylindrical homology}

With the reasonable assumption in mind that the topological recursion relations in Floer homology also hold true in (positive) symplectic homology and hence also carry over to non-equivariant cylindrical contact homology, we now formulate our main theorem. Since we assumed that there are no holomorphic planes in $\IR\times V$ and hence the usual Gromov compactness result holds we define the Gromov-Witten potential $\If$ of a stable Hamiltonian manifold as the part of the rational SFT Hamiltonian $\Ih$ of $V$ counting holomorphic spheres without punctures, $\If = \Ih|_{p=0=q}$. Note that in the contact case this agrees with the Gromov-Witten potential of a point due to the maximum principle and is determined by the Gromov-Witten potential of the symplectic fibre in the case when the stable Hamiltonian manifold is a symplectic mapping torus as every holomorphic map $\CP\to\IR\times M_{\phi}\to \IR\times S^1\cong \IC^*$ is constant by Liouville's theorem.

\begin{theorem}\label{TRR-noneqCH} For three \emph{different} non-generic special choices of coherent sections the following three \emph{topological recursion relations} hold in \emph{non-equivariant} cylindrical contact homology
\begin{itemize}
\item[(2,0):] $$ \del_{(\check{\alpha},i)} = \frac{\del^2 \If}{\del t^{\alpha,i-1}\del t^{\mu}} \eta^{\mu\nu} \del_{\check{\nu}}$$
\item[(1,1):] $$ N\,\del_{(\check{\alpha},i)} = \frac{\del^2 \If}{\del t^{\alpha,i-1}\del t^{\mu}} \eta^{\mu\nu} N\, \del_{\check{\nu}} + \frac{1}{2}[\del_{(\check{\alpha},i-1)}, \check{N}\,\del]_+ $$
\item[(0,2):] $$N(N-1)\,\del_{(\check{\alpha},i)} = \frac{\del^2 \If}{\del t^{\alpha,i-1}\del t^{\mu}} \eta^{\mu\nu} N(N-1)\, \del_{\check{\nu}} +[\del_{(\check{\alpha},i-1)},\check{N}(N-1)\,\del]_+$$
\end{itemize}
\vspace{0.5cm}
where $N:=\displaystyle{\sum_{\beta,j}}t^{\beta,j}\frac{\del}{\del t^{\beta,j}}$, $\check{N}:=\displaystyle{\sum_{\beta,j}} t^{\beta,j}\frac{\del}{\del \check{t}^{\beta,j}}$ and $[f,g]_+=f\circ g +(-1)^{\deg(f)deg(g)} g\circ f$ denotes a graded anti-commutator with respect to the operator composition. Notice that, $\check{N}$ and $\del$ having odd degree, in the above formulas the anti-commutator always corresponds to a sum.
\end{theorem} 

\vspace{0.5cm}

\subsection{Proof of the Main Theorem}
In this section we estabilish the psi-class localization result needed to prove Theorem \ref{TRR-noneqCH}, but also Proposition \ref{TRR-Floer}. We will show how coherent sections of tautological bundles on the moduli space of SFT-curves can be chosen such that their zero locus localizes on nodal configurations and boundary strata (multi-level curves).\\

Such localization will be the analogue, in presence of coherence conditions, of the usual result in Gromov-Witten theory describing the divisor $\psi_{i,r}=c_1(\LL_{i,r})$ on $\CM_{0,r,A}(X)$ as the locus of nodal curves where the $i$-th marked point lies on a different component with respect to a pair of other reference marked points. We will divide the discussion in three parts, corresponding to the three kind of topological recursion relations $(2,0)$, $(1,1)$ and $(0,2)$.\\

For the moment we stay general and consider the full SFT moduli space of curves with any number of punctures and marked points in a general manifold with stable Hamiltonian structure. In particular we will describe a special class of coherent collections of sections $s_{i,r}$ for the tautological bundles $\LL_{i,r}$ on the moduli spaces $\CM_{0,r,A}(\Gamma^+,\Gamma^-)$. We will then consider a sequence inside such class converging to a (no longer coherent) collection of sections whose zeros will completely be contained in the boundary strata (both nodal and multi-floor curves) of $\CM_{0,r,A}(\Gamma^+,\Gamma^-)$.\\

In \cite{FR} we already explained how to choose a (non-generic) coherent collection of sections $s_{i,r}$ in such a way that, considering the projection $\pi_r:\CM_{g,r,A}(\Gamma^+,\Gamma^-)\to\CM_{g,r-1,A}(\Gamma^+,\Gamma^-)$ consisting in forgetting the $r$-th marked point, the following \emph{comparison formula} holds for their zero sets:
\begin{equation}\label{comparison}
s_{i,r}^{-1}(0) = \pi_{r}^{-1}(s_{i,r-1}^{-1}(0))+D^\mathrm{const}_{i,r},
\end{equation}
The sum in the right hand side means union with the submanifold $D^\mathrm{const}_{i,r}$ of nodal curves with a constant sphere bubble carrying the $i$-th and $r$-th marked points, transversally intersecting 
$\pi_{r}^{-1}(s_{i,r-1}^{-1}(0))$.\\

We wish to stress the fact that such choice is possible because any codimension $1$ boundary of the moduli space $\CM_{g,r,A}(\Gamma^+,\Gamma^-)$ decomposes into a product of moduli spaces where the factor containing the $i$-th marked point carries the same well defined projection map $\pi_{r}$. This is because codimension $1$ boundary strata are always formed by non-constant maps, which remain stable after forgetting a marked point.\\

In fact coherence also requires that our choice of coherent collection of sections is symmetric with respect to permutations of the marked points (other than the $i$-th, carrying the descendant). We can reiterate this procedure until we forget all the marked points but the $i$-th, getting easily
\begin{equation}\label{sdd}
s_{i,r}^{-1}(0)=(\pi_{1}^*\circ\ldots\circ\hat{\pi}_i^*\circ\ldots\circ\pi_{r}^* \, s_{i,1})^{-1}(0) + \sum_{\substack{I\sqcup J=\{1,\ldots,r\}\\ \{i\}\subsetneq I \subseteq \{1,\ldots,r\}}} D^\mathrm{const}_{(I|J)}
\end{equation}
where $D^\mathrm{const}_{(I|J)}$ is the submanifold of nodal curves with a constant sphere bubble carrying the marked points labeled by indices in $I$. Such choice of coherent sections is indeed symmetric with respect to permutation of the marked points.\\

However, forgetting all of the marked points is not what we want to do in general, so we may take another approach, that does not specify whether the points we are forgetting are marked points or punctures. Forgetting punctures only makes sense after forgetting the map too.\\

Indeed, consider the projection $\sigma:\CM_{g,r,A}(\Gamma^+,\Gamma^-)\to \CM_{g,r+|\Gamma^+|+|\Gamma^-|}$ to the Deligne-Mumford moduli space of stable curves consisting in forgetting the holomorphic map and asymptotic markers, consequently stabilizing the curve, and considering punctures just as marked points. For simplicity, denote $n=r+|\Gamma^+|+|\Gamma^-|$. The tautological bundle $\LL_{i,r}$ on $\CM_{g,r,A}(\Gamma^+,\Gamma^-)$ coincides, by definition, with the pull-back along $\sigma$ of the tautological bundle $\LL_{i,n}$ on $\CM_{g,n}$ away from the boundary stratum $D_i\subset \CM_{g,r,A}(\Gamma^+,\Gamma^-)$ of nodal curves with a (possibly non-constant) bubble carrying the $i$-th marked point alone and the boundary stratum $D'_i\subset \CM_{g,r,A}(\Gamma^+,\Gamma^-)$ of multi-level curves with a level consisting in a holomorphic disk bounded by a Reeb orbit and carrying the $i$-th marked point.\\

At this point we are going to make the following assumption, which will hold throughout the paper.\\

{\bf Assumption:} In $V\times \IR$ there is no holomorphic disk bounded by a Reeb orbit. This implies, in particular, $D'_i=\emptyset$.\\

We choose now a coherent collection of sections $\tilde{s}_{i,n}$ on the Deligne-Mumford moduli space of stable curves $\CM_{g,n}$. The definition of such coherent collection is the same as for the space of maps, but this time we impose coherence on each real-codimension $2$ divisor of nodal curves (as opposed to the case of maps, where we only imposed coherence at codimension $1$ boundary strata). Such a coherent collection pulls back to a coherent collection of sections on $\CM_{g,r,A}(\Gamma^+,\Gamma^-)$ away from the already considered boundary stratum $D_i$ (the only one still present after the above assumption). We then scale such sections to zero smoothly as they reach $D_i$ using coherent cut-off functions (as we did for the comparison formula in \cite{FR}) getting this way a coherent collection of sections on the full $\CM_{g,r,A}(\Gamma^+,\Gamma^-)$. Precisely as in Gromov-Witten theory (see e.g. \cite{G}) we can see easily that the zero we added at $D_i$ has degree $1$.\\

Once more, suche construction is possible because any codimension $1$ boundary of the moduli space $\CM_{g,r,A}(\Gamma^+,\Gamma^-)$ decomposes into a product of moduli spaces where the factor containing the $i$-th marked point carries the same well defined projection map $\sigma$. This is because codimension $1$ boundary strata are always formed by multi-level curves, each level carrying at least two punctures (by the above assumption) and the $i$-th marked point, hence remaining stable after forgetting the map.\\

We then get, on $\CM_{g,r,A}(\Gamma^+,\Gamma^-)$,
\begin{equation}\label{loc1}
s_{i,r}^{-1}(0)=(\sigma^* \, \tilde{s}_{i,n})^{-1}(0) + D_i
\end{equation}

This construction of a coherent collection of sections for $\CM_{g,r,A}(\Gamma^+,\Gamma^-)$ moves the problem of explicitly describing their zero locus to the more tractable space of curves $\CM_{g,n}$. Notice first of all that, since $\CM_{g,n}$ has no (codimension $1$) boundary, any generic choice (coherent or not) of sections for the tautological bundles will give rise to a zero loci with the same homology class. However, when we pull-back such section via $\sigma$, we want to remember more than just the homology class of its zero (as required by coherence), so we need to make some specific choice.\\

Let us now restrict ourselves to genus $0$. In Gromov-Witten theory, where there is no need for coherence conditions, we are used to select two marked points besides the one carrying the psi-class and succesively forget all the other ones until we drop on $\CM_{0,3}=\mathrm{pt}$, where the tautological bundle $\LL_{i,3}$ is trivial. This approach is not possible in our SFT context where we require coherence on $\CM_{g,n}$. Indeed, if we select two punctures labeled by $j$ and $k$, we automatically lose the required symmetry with respect to permutation of the marked points. To overcome this problem we need to use coherent collections of multisections (whose image is a branched manifold) in order to average over all the possible choices of a pair of punctures.\\

Let us choose an averaged (over all the possible ways of choosing $2$ marked points out of $n$) multi-section for $\LL_{i,n}$ on $\CM_{0,n}$ such that its zero locus has the (averaged) form
\begin{equation}\label{average}
s_{i,n}^{-1}(0)=\frac{(n-3)!}{(n-1)!}\sum_{\substack{2\leq k\leq n-2\\ I\sqcup J=\{1,\ldots,n\}\\i\in I,|J|=k}} \frac{k!}{(k-2)!} D_{(I|J)}
\end{equation}

It is easy to see that such a (holomorphic, non-generic) section can be perturbed to a generic smooth section which is also coherent (this is in fact a statement about all of the sections $s_{i,j}$ together, for $3\leq j\leq n$). The zero locus of such (multi)section will form a (branched) codimension $2$ locus in the tubular neighborhood of the unperturbed zero locus $s_{i,j}^{-1}(0)$, transversally intersecting such locus. For notational and visualization simplicity we will analyze the case of $\CM_{0,5}$ in the example below, leaving the general case as an exercise for the reader. Once such coherent collection of sections $\tilde{s}_{i,j}$ is constructed we consider a sequence of sections $\tilde{s}^{(k)}_{i,j}$ with $\tilde{s}^{0}_{i,j}=\tilde{s}_{i,j}$ and converging back to the old non-generic $s_{i,j}$ as $k\to\infty$. This limit construction determines a codimension $2$ locus in the moduli space $\CM_{0,r}(\Gamma^+,\Gamma^-)$ completely contained in the boundary strata formed my nodal and multilevel curves, corresponding to the first summand in the right hand side of equation (\ref{loc1}). The explicit form of the boundary components involved by such locus is described by formula (\ref{average}), where the divisor $D_{(I|J)}$ refers to the source nodal Riemann surfaces for the multilevel curves in the target $V\times \IR$.

\begin{example}\label{DMcoherence}
This example, and the understanding of the general phenomenon it describes, emerged in a discussion with Dimitri Zvonkine. Consider the moduli space $\CM_{0,5}$, whose boundary divisors, formed by nodal curves, we denote $D_{ij}$, $1\leq i < j \leq 5$, where $D_{ij}$ is the space of nodal curves with a bubble carrying the $i$-th and $j$-th points alone. The intersection structure of such divisors is represented by the following picture.
\begin{center}
\includegraphics[width=9cm]{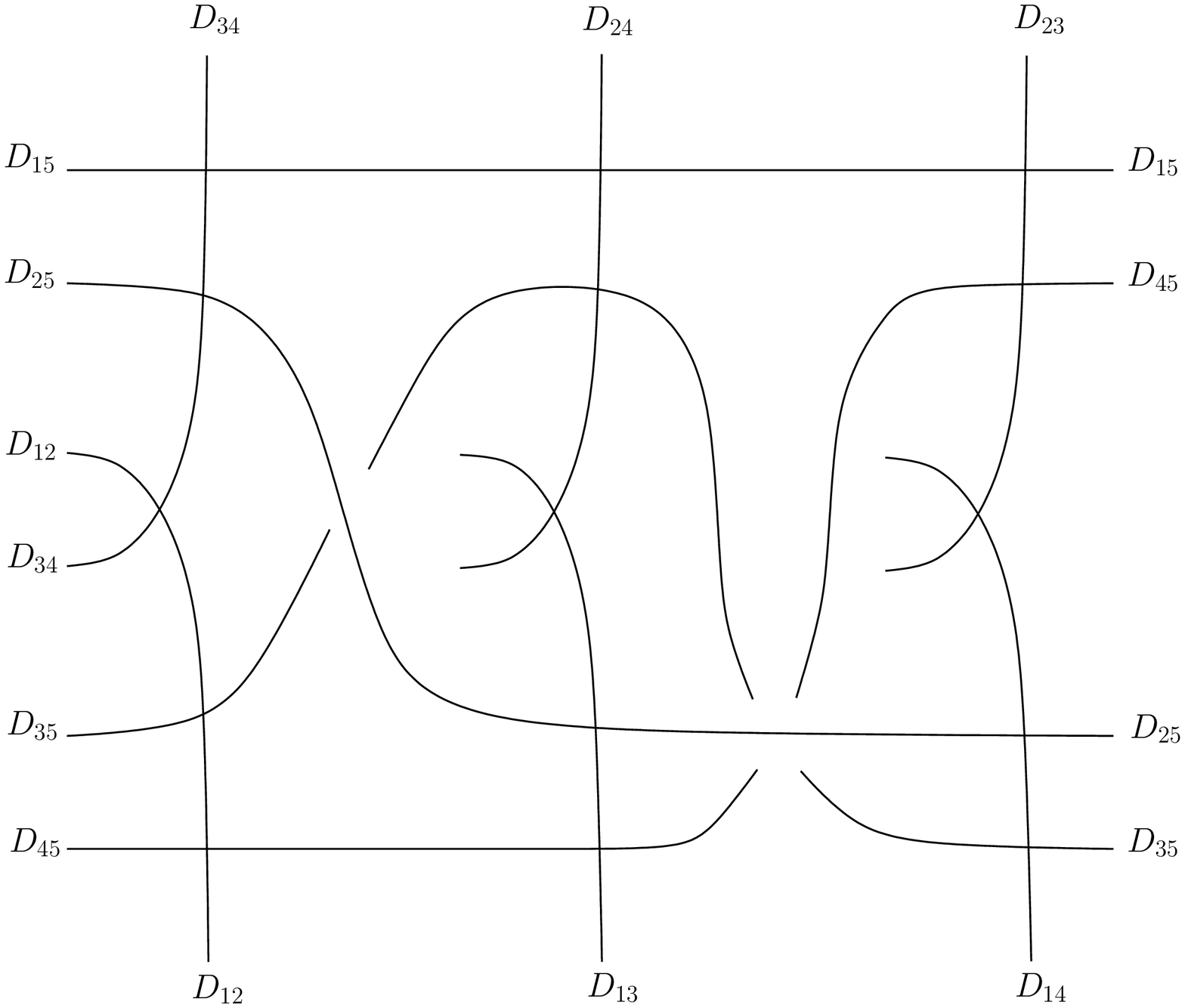}
\end{center}
When two different nodal divisors intersect, they do it with indersection index $+1$. The self-intersection index of any of them is, on the other hand, $-1$. Each of the $D_{ij}$ being a copy of $\mathbb{P}^1$ (representing a copy of the moduli space $\CM_{0,4}$ appearing at the boundary of $\CM_{0,5}$), this means that the normal bundle $N_{D_{ij}}$ of such $D_{ij}$ has Chern class $c_1(N_{D_{ij}})=-1$, hence the tubular neighborhood of $D_{ij}$ is a copy of $\tilde{\mathbb{C}}^2$, i.e. $\mathbb{C}^2$ blown up at $0$, $D_{ij}$ itself being the exceptional divisor. An intersecting $D_{kl}$ can then be seen as a line through the origin of $\tilde{\mathbb{C}}^2$.\\

Consider now the tautological line bundle $\LL_{1,5}$ on $\CM_{0,5}$. Using formula (\ref{average}), the corresponding averaged psi-class, i.e. the (dual to the) zero locus of a averaged multi-section of $\LL_{1,5}$, has the form
$$s_{1,5}^{-1}(0)=\frac{1}{2}(D_{12}+D_{13}+D_{14}+D_{15})+\frac{1}{6}(D_{23}+D_{24}+D_{25}+D_{34}+D_{35}+D_{45})$$
We now want to perturb such multi-section $s$ to a coherent multi-section $\tilde{s}$ by a small perturbation int he neighborhood of the nodal divisors. In fact it will be sufficient to describe how the zero locus is perturbed.\\

First notice that, once the line bundle $\LL_{1,5}$ is chosen, the nodal divisors $D_{ij}$ are split into two different sets, namely the ones for which $i=1$ and the ones with $i\neq 1$ (appearing in the first and second summand in the above averaged formula). The perturbation will be symmetric with respect to permutations inside these two subsets separately. Looking at the picture above for visual help, let us start by perturbing $\frac{1}{6}D_{34}$ away from itself in such a way that it still intersects $D_{34}$ at $D_{34}\cap D_{15}$ with index $-\frac{1}{6}$, at $D_{34}\cap D_{25}$ with index $+\frac{1}{6}$ and at $D_{34}\cap D_{12}$ with index $-\frac{1}{6}$ (notice that the total self-intersection index is $-\frac{1}{6}$, as it should for $\frac{1}{6}D_{34}$). The analogous choice is to be made for each of the divisors $D_{ij}$ with $i \neq 1$. Then we perturb $\frac{1}{2}D_{15}$ away from itself in such a way that it still intersects $D_{15}$ at $D_{15}\cap D_{34}$, $D_{15}\cap D_{24}$ and $D_{15}\cap D_{23}$ always with intersection index $-\frac{1}{6}$ (summing to a total self-intesection index of $-\frac{1}{2}$, as it should be for $\frac{1}{2}D_{15}$). This is in fact a multisection of the normal bundle $N_{D_{15}}$ formed by superimposing three sections of weight $\frac{1}{6}$ each, having a zero (of index $-\frac{1}{6}$) at  $D_{15}\cap D_{34}$, $D_{15}\cap D_{24}$ and $D_{15}\cap D_{23}$ respectively. Notice that such perturbation of $\frac{1}{2}D_{15}$ still intersects $D_{34}$ in a punctured neighborhood of $D_{34}\cap D_{15}$ with total intersection index $\frac{2}{6}$ and once more precisely at $D_{34}\cap D_{15}$ with intersection index $\frac{1}{6}$. The analogous choice is to be made for all of the divisors $D_{1j}$. See the lef side of next picture for some intuition.

\begin{center}
\includegraphics[width=10cm]{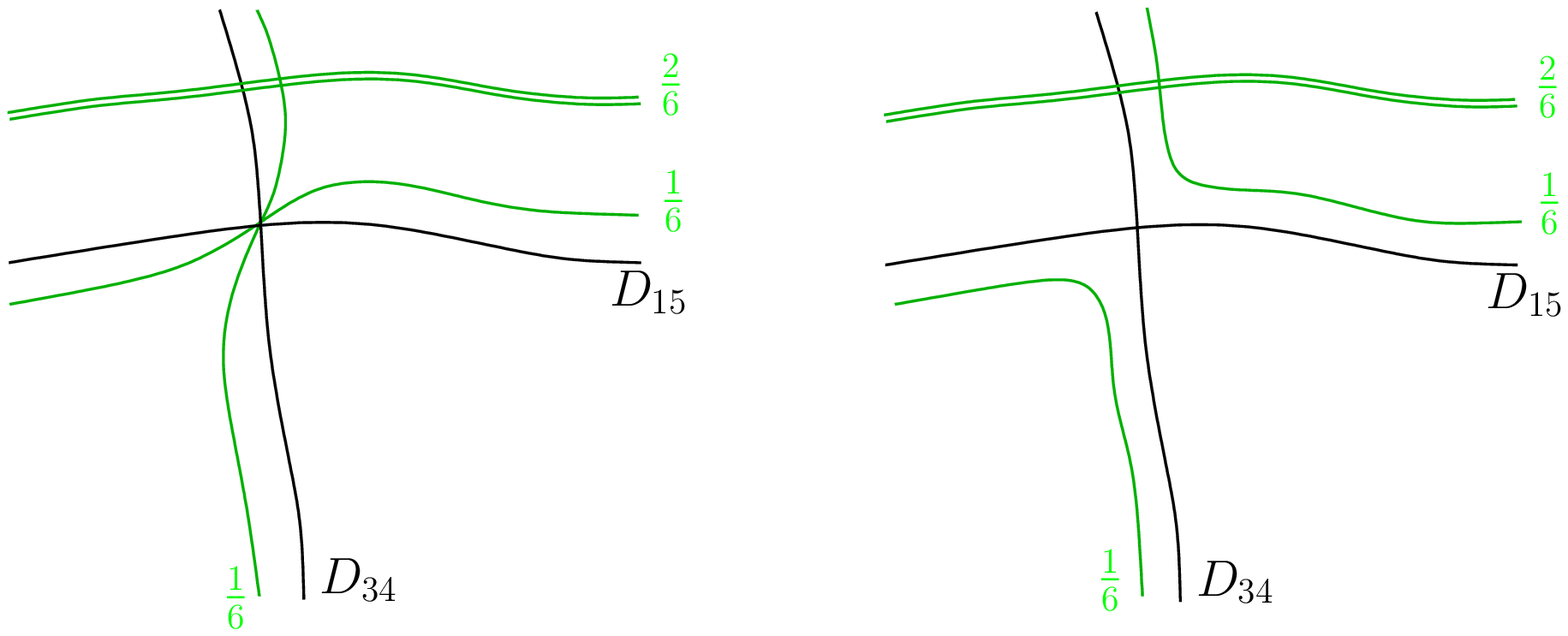}
\end{center}

At the point $D_{34}\cap D_{15}$ we will combine the two perturbed divisors to annihilate their intersection there (their intersection indices being $-\frac{1}{6}$ and $+\frac{1}{6}$). This gives rise to a hyperbolic (and smooth, as we want to get generic sections) behaviour of the zero locus that will now avoid $D_{15}$ completely (right side of the above picture). In a tubular neighborhood of $D_{34}$ the situation is described by the following picture, representing such neighborhood as $\tilde{\mathbb{C}}^2$.

\begin{center}
\includegraphics[width=10cm]{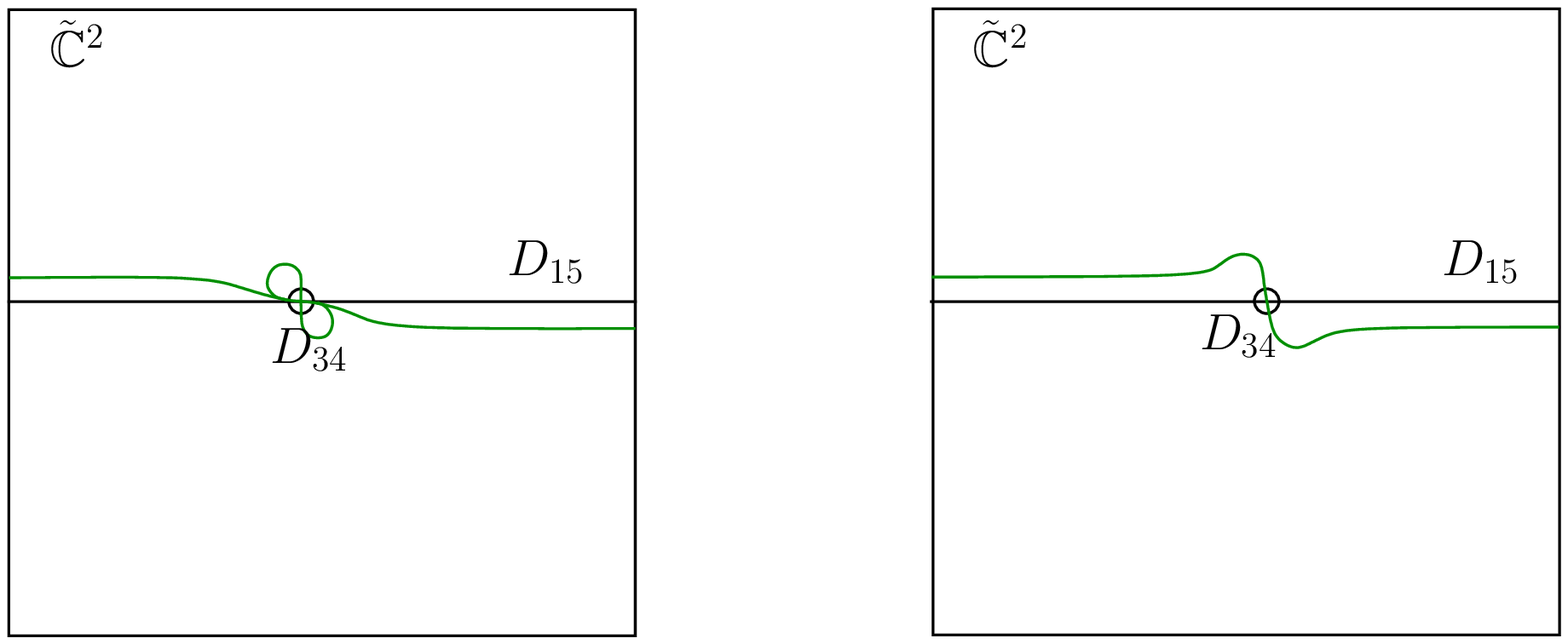}
\end{center}

The circle at the origin is the exceptional divisor $D_{34}$ and points on it are identified along diameters. Notice that, in the second picture, we see the hyperbolic behaviour and the fact that the green zero locus does avoid $D_{15}$.\\

We are now ready to prove that such averaged perturbed section is coherent with the corresponfing section on $\CM_{0,4}$, whose zero locus, using once more formula (\ref{average}), will be
$$s_{1,4}^{-1}(0)=\frac{1}{3}(D_{12}+D_{13}+D_{14}).$$
Indeed, the perturbed zero locus does not intersect at all $D_{1k}$, coherently with the fact that $\LL_{1,5}$ pulls back to the trivial bundle at such divisors, while at each of the $D_{ij}$ with $i\neq 1$ the situation is the same as for $D_{34}$: two of the three branches of the multisection of $N_{D_{15}}$, each with weight $\frac{1}{6}$, intersect it close to $D_{34}\cap D_{15}$ (total index $\frac{2}{6}$), and similarly for $N_{D_{12}}$, while, close to $D_{34}\cap D_{25}$, both the perturbation of $\frac{1}{6}D_{34}$ and the perturbation to $\frac{1}{6}D_{25}$ intersect $D_{34}$ (total index $\frac{2}{6}$). This is coherent with the above averaged formula for $s_{1,4}^{-1}(0)$.\\

With the very same approach (only the combinatorics being more complicated) we can trat the general case of $\CM_{0,n}$, constructing the analogous perturbation to the averaged formula (\ref{average}) for the psi-class on the Deligne-Mumford space of curves.
\end{example}

In order to prove Theorem \ref{TRR-noneqCH} we will make three different choices of special coherent collections of sections. Indeed, for equation $(2,0)$, the idea is not remembering any marked point, but only averaging with respect to the possible choices of two punctures. In this case we can choose an averaged coherent collection of multisections on the Deligne-Mumford moduli space of curves with $|\Gamma^+|+|\Gamma^-|+1$ marked points (using the perturbation technique of example \ref{DMcoherence}, where we are keeping all the puctures and the $i$-th marked point, carrying the psi-class), and then use equation (\ref{sdd}) to go to the space of maps $\CM_{0,r,A}(\Gamma^+,\Gamma^-)$. This coherent collection is evidently symmetric, with respect to permutations of marked points and punctures separately. Its zero locus, in the moduli space $\CM_{0,r,A}(\Gamma^+,\Gamma^-)$, has the form
$$\frac{P(P-1)}{2} s_{i,r,\Gamma^+,\Gamma^-}^{-1}(0) = \sum_{\substack{i\in I,\ I\sqcup J=\{1,\ldots,r\}\\ \Gamma^+_1 \sqcap \Gamma^+_2 = \Gamma^+ \\ \Gamma^-_1 \sqcap \Gamma^-_2 = \Gamma^- \\ P_2=|\Gamma^+_2|+|\Gamma^-_2|}} \frac{P_2(P_2-1)}{2}\  D_{(I,\Gamma^+_1,\Gamma^-_1|J,\Gamma^+_2,\Gamma^-_2)}$$
where $P=|\Gamma^+|+|\Gamma^-|$ and $D_{(I,\Gamma^+_1,\Gamma^-_1|J,\Gamma^+_2,\Gamma^-_2)}$ refers to a codimension $2$ locus in the tubular neighborhood of two-components curves joint at a node or puncture (hence nodal or two-level) where marked points and puctures split on each component as indicated by the subscript. The combinatorial factor on the left-hand side accounts for the possible ways of choosing two punctures to be remembered out of $P$, while the one on the right-hand side accounts for the number of ways a term of the form $D_{(I,\Gamma^+_1,\Gamma^-_1|J,\Gamma^+_2,\Gamma^-_2)}$ appears in the described averaging construction.\\

As a second possible choice we will start from an averaged coherent collection of multisections on the Deligne-Mumford moduli space of curves with $|\Gamma^+|+|\Gamma^-|+2$ marked points (like in example \ref{DMcoherence} where, again, we keep all the punctures, the $i$-th marked point carrying the psi-class, but also another marked point). There are exactly $r-1$ different forgetful projections from the space $\CM_{0,r,A}(\Gamma^+,\Gamma^-)$ to such $\CM_{0,|\Gamma^+|+|\Gamma^-|+2}$, corresponding to the numbering of the extra remembered marked point (excluding the $i$-th, which is also always remebered). In order to obtain a coherent collection on $\CM_{0,r,A}(\Gamma^+,\Gamma^-)$ which is also symmetric with respect to permutations of the marked points we need to use the pull-back construction of equation (\ref{sdd}), but also average among these different forgetful projections. After some easy combinatorics, its zero locus has the form
\begin{equation*}
\begin{split}
&(r-1)\frac{P(P+1)}{2} s_{i,r,\Gamma^+,\Gamma^-}^{-1}(0) =\\
&\sum_{\substack{i\in I,\ I\sqcup J=\{1,\ldots,r\}\\ |I|=r_1,\ |J|=r_2\\ \Gamma^+_1 \sqcap \Gamma^+_2 = \Gamma^+ \\ \Gamma^-_1 \sqcap \Gamma^-_2 = \Gamma^- \\ P_2=|\Gamma^+_2|+|\Gamma^-_2|}} \left[r_2 \frac{P_2(P_2+1)}{2}+(r_1-1)\frac{P_2(P_2-1)}{2}\right]\  D_{(I,\Gamma^+_1,\Gamma^-_1|J,\Gamma^+_2,\Gamma^-_2)}
\end{split}
\end{equation*}

Finally, as a last choice, we start from the moduli space of curves with $|\Gamma^+|+|\Gamma^-|+3$ marked points, so that we are keeping, after forgetting the map, two extra marked points (besides the $i$-th). This time there will be exactly $\frac{(r-1)(r-2)}{2}$ forgetful maps from $\CM_{0,r,A}(\Gamma^+,\Gamma^-)$ to $\CM_{0,|\Gamma^+|+|\Gamma^-|+3}$, corresponding to the numbering of the two extra remembered marked points. The pull-back construction of equation (\ref{sdd}) needs then to be averaged among these possible choices in order to be symmetric with respect to permutations of marked points. This time the averaging combinatorics gives

\begin{equation*}
\begin{split}
&\frac{(r-1)(r-2)}{2}\frac{(P+2)(P+1)}{2} s_{i,r,\Gamma^+,\Gamma^-}^{-1}(0) =\\
&\sum_{\substack{i\in I,\ I\sqcup J=\{1,\ldots,r\}\\ |I|=r_1,\ |J|=r_2\\ \Gamma^+_1 \sqcap \Gamma^+_2 = \Gamma^+ \\ \Gamma^-_1 \sqcap \Gamma^-_2 = \Gamma^- \\ P_2=|\Gamma^+_2|+|\Gamma^-_2|}} \left[\frac{r_2(r_2-1)}{2} \frac{(P_2+2)(P_2+1)}{2}+ r_2(r_1-1)\frac{P_2(P_2+1)}{2}\right.\\
&\phantom{\sum_{\substack{i\in I,\ I\sqcup J=\{1,\ldots,r\}\\ |I|=r_1,\ |J|=r_2\\ \Gamma^+_1 \sqcap \Gamma^+_2 = \Gamma^+ \\ \Gamma^-_1 \sqcap \Gamma^-_2 = \Gamma^- \\ P_2=|\Gamma^+_2|+|\Gamma^-_2|}}}\left.+\frac{(r_1-1)(r_1-2)}{2}\frac{P_2(P_2-1)}{2}\right]\ D_{(I,\Gamma^+_1,\Gamma^-_1|J,\Gamma^+_2,\Gamma^-_2)}
\end{split}
\end{equation*}

Such three different choices of multisections can also be superimposed to form further multisections which are, of course, still coherent. This corresponds to some linear combination of the equations for their zero loci.  In particular, by taking respectively the first one, the second one minus the first one, and the third one minus twice the second one plus the first one, we get multisections whose wero loci have the form
$$\frac{P(P-1)}{2}\  s_{i,r,\Gamma^+,\Gamma^-}^{-1}(0) = \sum \frac{P_2(P_2-1)}{2}\  D_{(I,\Gamma^+_1,\Gamma^-_1|J,\Gamma^+_2,\Gamma^-_2)}$$
$$(r-1)P\ s_{i,r,\Gamma^+,\Gamma^-}^{-1}(0) = \sum r_2P_2\  D_{(I,\Gamma^+_1,\Gamma^-_1|J,\Gamma^+_2,\Gamma^-_2)}$$
$$\frac{r-1(r-2)}{2}\ s_{i,r,\Gamma^+,\Gamma^-}^{-1}(0) = \sum \frac{r_2(r_2-1)}{2}\  D_{(I,\Gamma^+_1,\Gamma^-_1|J,\Gamma^+_2,\Gamma^-_2)}$$
\\

To complete the proof of Theorem \ref{TRR-noneqCH} we just need to notice that the limit procedure taking the perturbed section $\tilde{s}_{i,n}$ of $\LL_{i,n}$ back to its original non-generic limit $s_{i,n}$ indeed amounts to selecting, in the space of maps relevant for cylindrical non-equivariant contact homology, via formula (\ref{loc1}), either nodal configurations (and this is obvious), or two-level ones, i.e. where the two smooth components are connected by a puncture instead of a node. While the two-level curves are of codimension one and not two in the space of maps (indeed this extra dimension remembers the information on the angular coordinate used for the gluing at the connecting puncture), we correct this error by fixing the decoration, i.e. the identification of the tangent planes at both points, \emph{a priori} as follows.\\

Since each of the two cylinders connected by the puncture carries one (or two in the case of (0,2)) of the remembered additional marked points, the position of these additional marked points can be used to fix unique $S^1$-coordinates and hence asymptotic markers on each of the two cylinders, which in turn defines a natural decoration by simply requiring that the two asymptotic markers are identified. Note that this turns the additional markers used for fixing the $S^1$-coordinates automatically into additional marked points constrained to $\IR\times \{0\}$. The following picture illustrate such phenomenon for the case relevant to equation $(1,1)$: the red puncture and marked point are those we are remembering, the black marked point is the one carrying the psi-class, whose power is specified by the index, the dashed line represents $\IR\times \{0\}$ and the green arrow indicates the matching condition between the two dashed lines (notice that, for simplicity, we are not explicitly drawing any other marked point). The $(0,2)$ case is completely similar, only involving averaging between the two possible choices of marked point to be constrained to $\IR\times \{0\}$.\\

\begin{center}
\includegraphics[width=10cm]{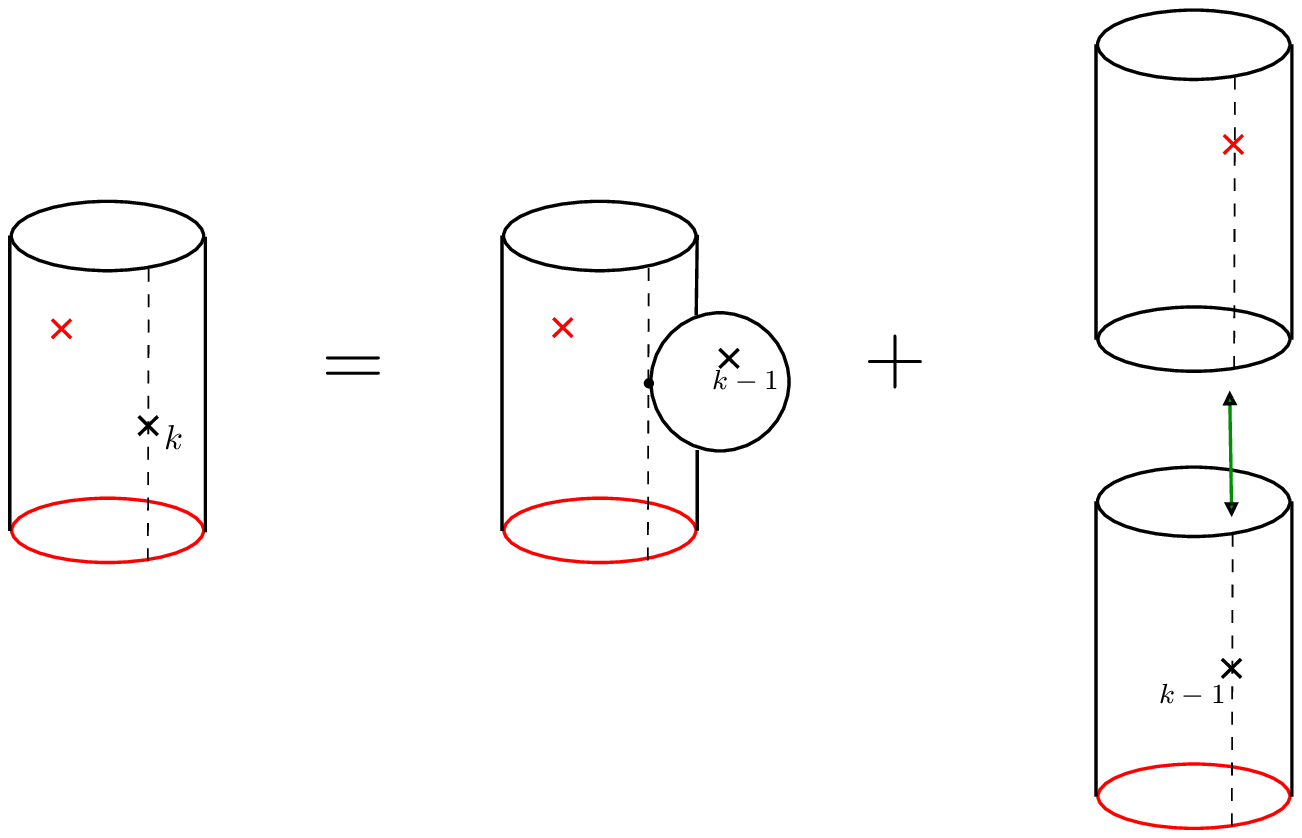}
\end{center}

We are only left with translating all this geometric picture and the three above equations for the zero loci back into our generating funtions language for the potential (recall the definition of the non equivariant cylindrical homology differential). There we see that the averaging combinatorial coefficients in formula (\ref{average}) are absorbed in the right way to give rise to the statement of Theorem \ref{TRR-noneqCH}.

\vspace{0.5cm}

\section{Applications}

In this final section we want to apply the topological recursion result for non-equivariant cylindrical homology to (equivariant) cylindrical homology. As an important result we show that, as in rational Gromov-Witten theory, all descendant invariants can be computed from primary invariants, i.e. those without descendants. Furthermore we will prove that the topological recursion relations imply that one can define an action of the quantum cohomology ring $QH^*(V)$ of the target manifold (defined using the Gromov-Witten potential $\If$ of $V$ introduced above) on the non-equivariant cylindrical homology $HC^{\textrm{non-}S^1}_*(V)$ by counting holomorphic cylinders with one constrained marked point. 

\vspace{0.5cm}

\subsection{Topological recursion in cylindrical homology}
Since the chain space for non-equivariant cylindrical homology splits, $C^{\textrm{non-}S^1}_*=\hat{C}_*\oplus\check{C}_*$, it follows that the linear maps on the chain space, obtained by differentiating the differential of non-equivariant cylindrical homology with respect to $t^{\alpha,p}$- or $\check{t}^{\alpha,p}$-variables, can be restricted to linear maps between $\hat{C}_*$ and $\check{C}_*$, respectively. On the other hand, since each of the spaces $\hat{C}_*$ and $\check{C}_*$ is just a copy of the chain space for (equivariant) cylindrical homology, with degree shifted by one for the second space, $\hat{C}_* = C_*$, $\check{C}_*=C_*[1]=C_{*+1}$, we can translate the linear maps from non-equivariant cylindrical homology to (equivariant) cylindrical homology as follows. \\
  
While the restricted linear maps $\del_{(\alpha,p)}: \hat{C}_*\to\hat{C}_*$ and $\del_{(\alpha,p)}: \check{C}_*\to\check{C}_*$ indeed agree with the linear maps $\del_{(\alpha,p)}: C_*\to C_*$ from cylindrical homology as defined in subsection 2.6, note that one can now introduce new linear maps $\del_{(\check{\alpha},p)}: C_*\to C_*$ on cylindrical homology by requiring that they agree with the linear maps $\del_{(\check{\alpha},p)}: \hat{C}_*\to \hat{C}_*$ (and hence $\del_{(\check{\alpha},p)}: \check{C}_*\to \check{C}_*$) from non-equivariant cylindrical homology.  \\

On the other hand, while the topological recursion relations we proved for the non-equivariant case are useful to compute the linear maps $\del_{(\check{\alpha},p)}$ on $HC^{\textrm{non-}S^1}_*$, the goal of topological recursion in cylindrical contact homology (as in rational SFT) is to compute the linear maps $\del_{(\alpha,p)}: HC_*\to HC_*$. In order to apply our results of the non-equivariant case to the equivariant case, we make use of the fact that (apart from the mentioned equivalence with $\del_{(\alpha,p)}: \hat{C}_*\to\hat{C}_*$ and $\del_{(\alpha,p)}: \check{C}_*\to\check{C}_*$) the linear map $\del_{(\alpha,p)}: C_*\to C_*$ also agrees with the restricted linear map $\del_{(\check{\alpha},p)}: \hat{C}_*\to\check{C}_*$. \\

In order to see this, observe that, while in the case of $\del_{(\alpha,p)}: \hat{C}_*\to\hat{C}_*$ (or $\del_{(\alpha,p)}: \check{C}_*\to\check{C}_*$) the free $S^1$-coordinate on the cylinder is fixed by the critical point on the negative (or positive) closed Reeb orbit, in the case of $\del_{\check{\alpha},p}: \hat{C}_*\to\check{C}_*$ the free $S^1$-coordinate on the cylinder is fixed by the additional marked point (and thereby turning it into a constrained marked point).  \\

With this we can prove the following corollary about topological recursion in (equivariant) cylindrical homology. 

\begin{corollary}\label{TRR-eqCH} For three \emph{different} non-generic special choices of coherent sections the following three \emph{topological recursion relations} hold in (equivariant) cylindrical contact homology
\begin{itemize}
\item[(2,0):] $$ \del_{(\alpha,i)} = \frac{\del^2 \If}{\del t^{\alpha,i-1}\del t^{\mu}} \eta^{\mu\nu} \del_{\nu}$$
\item[(1,1):] $$ N\,\del_{(\alpha,i)} = \frac{\del^2 \If}{\del t^{\alpha,i-1}\del t^{\mu}} \eta^{\mu\nu} N\, \del_{\nu} + \frac{1}{2}[\del_{(\alpha,i-1)}, \check{N}\,\del]_+ + \frac{1}{2}[\del_{(\check{\alpha},i-1)}, N\,\del]_+$$
\item[(0,2):] \begin{eqnarray*} N(N-1)\,\del_{(\alpha,i)} = \frac{\del^2 \If}{\del t^{\alpha,i-1}\del t^{\mu}} \eta^{\mu\nu} N(N-1)\, \del_{\nu} &+&[\del_{(\alpha,i-1)},\check{N}(N-1)\,\del]_+\\ &+&[\del_{(\check{\alpha},i-1)},N(N-1)\,\del]_+
\end{eqnarray*}
\end{itemize}
\end{corollary} 

\begin{proof}
While the relation (2,0) is immediately follows by identifying the linear map $\del_{(\alpha,p)}: C_*\to C_*$ with the restricted linear map $\del_{(\check{\alpha},p)}: \hat{C}_*\to\check{C}_*$, for the relations (1,1) and (0,2) it suffices to observe that 
\begin{eqnarray*} 
  (\del_{(\check{\alpha},i-1)}\circ\check{N}\del:\hat{C}_*\to\check{C}_*)  
 &=& (\del_{(\check{\alpha},i-1)}: \hat{C}_*\to\check{C}_*) \circ (\check{N}\del: \hat{C}_*\to\hat{C}_*) \\
 &+& (\del_{(\check{\alpha},i-1)}: \check{C}_*\to\check{C}_*) \circ (\check{N}\del: \hat{C}_*\to\check{C}_*) \\
 &=& (\del_{(\alpha,i-1)}: C_*\to C_*) \circ (\check{N}\del: C_*\to C_*) \\
 &+& (\del_{(\check{\alpha},i-1)}: C_*\to C_*) \circ (N\del: C_*\to C_*).
\end{eqnarray*}
\end{proof}

While it follows that the second and the third topological recursion relation involve the linear maps $\del_{\check{\alpha},p}: C_*\to C_*$ defined using non-equivariant contact homology and hence leave the frame of standard (equivariant) cylindrical homology, it is notable that the first topological recursion relation (2,0) indeed has the following important consequence. 

\begin{corollary} All linear maps $\del_{(\alpha,p)}: HC_*(V)\to HC_*(V)$ on cylindrical homology involving gravitational descendants can be computed from the linear maps $\del_{\alpha}: HC_*(V)\to HC_*(V)$  with no gravitational descendants and the primary rational Gromov-Witten potential of the underlying stable Hamiltonian manifold, i.e. again involving no gravitational descendants. \end{corollary}

\begin{proof}
For the proof it suffices to observe that after applying the topological recursion relation (2,0) the marked point with the descendant sits on the attached sphere, so that the linear maps with descendants can indeed be computed from the linear maps without descendants and the rational Gromov-Witten potential of the target manifold with gravitational descendants. Together with the  standard result of rational Gromov-Witten theory (generalized in the obvious way from symplectic manifolds to stable Hamiltonian manifolds without holomorphic planes) that the full descendant potential can be computed from the primary potential involving no descendants using the above mentioned topological recursion relations together with the divisor (to add more marked points on non-constant spheres), string and dilaton (for the case of constant spheres) equations, it follows the remarkable result that also in cylindrical homology the descendant invariants are determined by the primary invariants, that is, if we additionally include the primary Gromov-Witten potential.
\end{proof}

\begin{remark} Note that the first topological relation actually descends to homology, i.e. it holds for $\del_{(\alpha,i)}$ and $\del_{\nu}$ \emph{viewed as linear maps on cylindrical homology} $HC_*(V)$. In particular, while on the chain level all topological recursion relations only hold true for (three different) \emph{special} choices of coherent sections, \emph{after} passing to homology the first relation (2,0) holds for \emph{all} coherent sections. \end{remark}

As we already remarked, it follows from the maximum principle that the Gromov-Witten potential of a contact manifold simply agrees with the Gromov-Witten potential of a point. Since in this case it follows from dimensional reasons that after setting all $t$-variables to zero we have $$\frac{\del^2 \If}{\del t^{\alpha,i-1}\del t^{\mu}}|_{t=0} = 0,\; i>0,$$ we have the following important vanishing result for contact manifolds. \\

For the rest of this subsection as well as the next one we will restrict ourselves to the case where all formal $t$-variables are set to zero. \\

Following the notation in \cite{F2} and \cite{FR} let us denote by $HC^0_*(V)=H(C^0_*,\del^0)$ the cylindrical homology without additional marked points and hence without $t$-variables, which is obtained from the big cylindrical homology complex $HC_*(V)=H(C_*,\del)$ by setting all $t$-variables to zero. In the same way let us introduce the corresponding linear map $\del^1_{(\alpha,p)}: HC^0_*(V)\to HC^0_*(V)$ obtained again by setting $t=0$ and which now counts holomorphic cylinders with just one additional marked point (and descendants).   

\begin{corollary} In the case when $V$ is a contact manifold, after setting all $t$-variables to zero, the corresponding descendant linear maps $\del^1_{(\alpha,p)}: HC^0_*(V)\to HC^0_*(V)$, $p>0$ are zero. \end{corollary}

While this result shows that counting holomorphic cylinders with one additional marked point and gravitational descendants is not very interesting in the case of contact manifolds, it is clear from our ongoing work on topological recursion in full rational symplectic field theory that the arguments used above do \emph{not} apply to the sequence of commuting Hamiltonians $\Ih^1_{(\alpha,p)}$ of rational SFT, which in the Floer case lead to the integrable hierarchies of Gromov-Witten theory. More precisely, we expect that the corresponding recursive procedure involves primary invariants belonging to a non-equivariant version of rational SFT. 

\vspace{0.5cm}

\subsection{Action of quantum cohomology on non-equivariant cylindrical homology}
As we already mentioned in subsection 3.2, in \cite{PSS} Piunikhin-Salamon-Schwarz defined an action of the quantum cohomology ring of the underlying symplectic manifold on the Floer (co)homology groups by counting Floer cylinders with one additional marked point constrained to $\IR\times\{0\}\subset\RS$. Note that for this the authors also needed to show that the concatination of two maps on Floer cohomology corresponds to the ring multiplication in quantum cohomology. \\

While in \cite{PSS} this result was proven by establishing appropriate compactness and gluing theorems for all appearing moduli spaces, in this final subsection we want to show how our topological recursion relation (1,1) together with the relation (2,0) can be used to define a corresponding action of the quantum cohomology (defined using the Gromov-Witten potential introduced above) on the non-equivariant cylindrical contact homology of a stable Hamiltonian manifold after setting all $t$-variables to zero. \\

In the same way as for closed symplectic manifolds we define the quantum cohomology $QH^*(V)$ of the stable Hamiltonian manifold $V$ as the vector space freely generated by formal variables $t^{\alpha}=t^{\alpha,0}$, with coefficients which are Laurent series in the $z_n$-variables. Note that, as vector spaces, the only difference to the usual cohomology groups $H^*(V)$ again lies in the different choice of coefficients. On the other hand, while for general stable Hamiltonian manifolds the quantum product defined using the Gromov-Witten three-point invariants is different from the usual product structure of $H^*(V)$, note that for contact manifolds we have $QH^*(V)=H^*(V)$ (with the appropriate choice of coefficients) as in this case the Gromov-Witten potential of $V$ agrees with that of a point. Recalling that the linear maps $\del^1_{\check{\alpha}}=\del^1_{(\check{\alpha},0)}$ actually descend to maps on (non-equivariant) cylindrical homology $HC^{0,\textrm{non-}S^1}_*(V)$, we prove the following  

\begin{corollary} The map
\begin{eqnarray*} 
QH^*(V)\otimes HC^{0,\textrm{non-}S^1}_*(V)\to  HC^{0,\textrm{non-}S^1}_*(V), 
&& (t^{\alpha},\hat{q}_{\gamma}) \mapsto \del^1_{\check{\alpha}}(\hat{q}_{\gamma}),\\
&& (t^{\alpha},\check{q}_{\gamma}) \mapsto \del^1_{\check{\alpha}}(\check{q}_{\gamma}),\\
\end{eqnarray*} 
defines an action of the quantum cohomology ring $QH^*(V)$ on the non-equivariant cylindrical homology $HC^{0,\textrm{non-}S^1}_*(V)$ (after setting all $t=0$). \end{corollary} 

\begin{proof}
It follows from our topological recursion relations (2,0) and (1,1) for non-equivariant cylindrical contact homology that, after setting all $t$-variables to zero, we indeed have the following two non-averaged topological recursion relations,
\begin{eqnarray*}
\del^2_{(\check{\alpha},i),(\beta,j)} &=& \frac{\del^2 \If}{\del t^{\alpha,i-1}\del t^{\mu}} \eta^{\mu\nu} \del^2_{\check{\nu},(\beta,j)} + \frac{\del^3 \If}{\del t^{\alpha,i-1}\del t^{\beta,j}\del t^{\mu}} \eta^{\mu\nu} \del^1_{\check{\nu}},\\
\del^2_{(\check{\alpha},i),(\beta,j)} &=& \frac{\del^2 \If}{\del t^{\alpha,i-1}\del t^{\mu}} \eta^{\mu\nu} \del^2_{\check{\nu},(\beta,j)} + \frac{1}{2}[\del^1_{(\check{\alpha},i-1)}, \del^1_{(\check{\beta},j)}]_+.
\end{eqnarray*}
While the first equation follows from differentiating the recursion relation (2,0) with respect to the formal variable $t^{\beta,j}$, the second equation follows from the recursion relation (1,1) by first setting all $t$-variables except $t^{\beta,j}$ to zero. \\

Ignoring invariance problems for the moment, the desired result follows by comparing both equations. Since the left side and the first summand on the right side of both equations agree, it follows that $$ \frac{1}{2}[\del^1_{(\check{\alpha},i-1)}, \del^1_{(\check{\beta},j)}]_+= \frac{\del^3 \If}{\del t^{\alpha,i-1}\del t^{\beta,j}\del t^{\mu}} \eta^{\mu\nu} \del^1_{\check{\nu}}.$$ On the other hand, since $[\del^1_{(\check{\alpha},i-1)}, \del^1_{(\check{\beta},j)}]_-=0$ on homology, and using the natural relation between the commutator and its corresponding anti-commutator, $\frac{1}{2}[\del^1_{(\check{\alpha},i-1)}, \del^1_{(\check{\beta},j)}]_+ + \frac{1}{2}[\del^1_{(\check{\alpha},i-1)}, \del^1_{(\check{\beta},j)}]_- = \del^1_{(\check{\alpha},i-1)}\circ\del^1_{(\check{\beta},j)}$,  it follows that after passing to homology we have $$\del^1_{(\check{\alpha},i-1)}\circ\del^1_{(\check{\beta},j)}= \frac{\del^3 \If}{\del t^{\alpha,i-1}\del t^{\beta,j}\del t^{\mu}} \eta^{\mu\nu} \del^1_{\check{\nu}},$$ so that the desired result follows after setting $i=1$ and $j=0$. \\

On the other hand, since the above recursion relations are true on the chain level only for special, in particular, two \emph{different} coherent collections of sections, the above reasoning leads to a true statement as the desired result is a result about invariants, i.e. it is independent of the chosen auxiliary data. While the above identity should hold after passing to homology, note that the invariance problem cannot be resolved as for the topological recursion relation (2,0) by simplying passing to homology, since the two equations with which we started involve linear maps counting holomorphic cylinders with more than one additional marked point. \\

In order to show that the desired composition rule still holds after passing to homology, we make use of the fact that we can choose nice coherent collections of sections interpolating between the two special coherent sections (in the sense of subsection 2.2) as follows. Since it follows from the proof of the main theorem in subsection 4.3 that our special coherent collections of sections are indeed pulled back from the moduli space of curves to the moduli space of maps (we can ignore the bubbles that we added afterwards here), we do not need to consider arbitrary interpolating coherent sections but only those which are again pull-backs of coherent sections on the underlying moduli space of curves. Since in the moduli space of curves (in contrast to the moduli space of maps) the strata of singular curves (maps) are of codimension at least two, we can further choose the homotopy such that it avoids all singular strata, so that all underlying curves in the homotopy are indeed smooth. Since we excluded holomorphic planes throughout the paper, it follows that the only singular maps that appear during the interpolation process are holomorphic maps where a cylinder without additional marked points splits off. But this implies that the difference between the two different special coherent collections of sections is indeed exact, so that the above equation indeed holds after passing to homology. 
\end{proof}

While in the same way we can give an alternative proof of the result of Piunikhin-Salamon-Schwarz by using our topological recursion relations (1,1) and (2,0) in symplectic Floer theory of subsection 3.3, in contrast note that in (equivariant) cylindrical homology, due to the differences in the topological recursion formulas in this case, neither $\del^1_{\alpha}$ nor $\del^1_{\check{\alpha}}$ defines an action of quantum cohomology on (equivariant) cylindrical homology. \\

Finally, using the isomorphism of Bourgeois-Oancea in \cite{BO}, we show how the latter result also establishes an action of the cohomology ring on the symplectic homology, and thereby generalizes the result of \cite{PSS} in the obvious way from closed manifolds to compact manifolds $X$ with contact boundary $\del X=V$. For the proof we assume that there not only no holomorphic planes in $(\IR\times) V$, but also no holomorphic planes in the filling $X$. Furthermore we assume that all $t$-variables are set to zero without explicitly mentioning it again.

\begin{corollary} Using the isomorphism between non-equivariant cylindrical homology and positive symplectic homology in \cite{BO}, our result defines an action of the cohomology ring $H^*(X)$ on (full) symplectic homology $SH^0_*(X)$ (at $t=0$).  \end{corollary}

\begin{proof}
Since we assume that there are no holomorphic planes in the filling $X$, it further follows from the computation of the differential in symplectic homology by Bourgeois and Oancea in \cite{BO} that the symplectic homology is given by the direct sum, 
$SH^0_*(X)=SH^{0,+}_*(X)\oplus H^{\dim X-*}(X)$. While the action of $H^*(V)$ on non-equivariant cylindrical contact homology $HC^{0,\textrm{non-}S^1}_*(V)$ established above defines an action of $H^*(X)$ on $SH^{0,+}_*(X)$ using the isomorphism $HC^{0,\textrm{non-}S^1}_*(V)\cong SH^{0,+}_*(X)$ and the natural map $H^*(X)\to H^*(V)$ defined by the inclusion $V\hookrightarrow X$, together with the natural action of $H^*(X)$ on itself we get the desired result. \end{proof}
 
Of course, we expect that this action agrees with the action of $H^*(X)$ on $SH^0_*(X)$ defined using the action on the Floer homology groups $FH^0_*(H)$ for admissible Hamiltonians and taking the direct limit, after generalizing the result in \cite{PSS} from closed symplectic manifolds to compact symplectic manifolds with contact boundary in the obvious way. 

\begin{example}
Using the natural map $H^*(Q)\to H^*(T^*Q)$ given by the projection, note that in the cotangent bundle case $X=T^*Q$ this defines an action of $H^*(Q)$ on $SH^0_*(T^*Q)$, which by \cite{AS} and \cite{SW} is isomorphic to $H_*(\Lambda Q)$, where $\Lambda Q$ denotes the loop space of $Q$. Introducing additional marked points on the cylinders in the proofs of Abbondandolo-Schwarz and Salamon-Weber, we expect that it can be shown that this action agrees with the natural action of $H^*(Q)$ on $H_*(\Lambda Q)$ given by the cap product and the base point map $\Lambda Q\to Q$. \\

On the other hand, as mentioned above, in the equivariant setting we do \emph{not} expect to find a natural action of $H^*(Q)$ on (equivariant) cylindrical homology $HC^0_*(S^*Q)$, which by \cite{CL} is isomorphic to the $S^1$-equivariant singular homology $H^{S^1}_*(\Lambda Q,Q)$. But this fits with the well-known fact that there is \emph{no} natural action of the cohomology \hbox{$(H^*(Q)\to) H^*(\Lambda Q)$} on relative $S^1$-equivariant homology $H^{S^1}_*(\Lambda Q,Q)$.
\end{example}

\subsection{Example: Cylindrical homology in the Floer case}

We end this paper by discussing briefly the important \emph{Floer case} of SFT, which was worked out in the paper \cite{F1} of the first author, including the neccessary transversality proof. Here $V=S^1\times M$ is equipped with a stable Hamiltonian structure $(\omega^H=\omega+dH\wedge dt,\lambda=dt)$ for some time-dependent Hamiltonian $H:S^1\times M\to\IR$ on a closed symplectic manifold $M=(M,\omega)$. It follows that the Reeb vector field is given by $R^H=\del_t+X^H_t$, so that in particular every one-periodic closed Reeb orbit is a periodic orbit of the time-dependent Hamiltonian $H$. More precisely, it can be shown, see \cite{F1}, that the chain complex of equivariant cylindrical contact homology naturally splits into subcomplexes generated by Reeb orbits of a fixed integer period and that the equivariant cylindrical homology generated by Reeb orbits of period one agrees with the standard symplectic Floer homology for the time-dependent Hamiltonian $H:S^1\times M\to\IR$. In order to see that the differentials indeed agree, one uses that the holomorphic map from the cylinder to $\IR\times V$ splits, $\tilde{u}=(h,u):\RS\to(\RS)\times M$, where the map $h:\RS\to\RS$ is just the identity (up to automorphisms of the domain). Note that in the Morse-Bott limit $H=0$  one arrives in the trivial circle bundle case and just gets back the relation between SFT and Gromov-Witten theory from \cite{EGH}.We now show the following important 

\begin{proposition} In the Floer case the topological recursion relations for equivariant cylindrical homology reproduce the topological recursion relations for symplectic Floer homology from section three. In particular, by passing to the Morse-Bott limit $H=0$, they reproduce the standard topological recursion relations of Gromov-Witten theory. Furthermore, the action of the quantum cohomology on the non-equivariant cylindrical homology splits and agrees with the action of quantum cohomology on symplectic Floer homology defined in \cite{PSS}. \end{proposition}

\begin{proof} In the same way as it follows from the fact that the map $h:\RS\to\RS$ is just the identity (up to automorphisms of the domain) that the differentials $\del$ in equivariant cylindrical homology and symplectic Floer homology naturally agree, it follows that  the linear maps $\del_{\check{\alpha}}$ for $\alpha\in H^*(M)$ introduced in symplectic Floer homology in section three and in equivariant cylindrical homology (using the corresponding map in non-equivariant cylindrical homology) agree. Furthermore it follows from the same result that for $\alpha_1=\alpha\wedge dt\in H^*(S^1\times M)$ we have
\[\del_{\alpha_1,p} = \del_{\check{\alpha},p}: C_*\to C_*,\;\;  \del_{\check{\alpha_1},p} = 0.\]
Note that the last equation follows from the fact that the $S^1$-symmetry is divided out twice. Again only working out the second topological recursion relation, it then indeed follows that  
\begin{eqnarray*} 
 (N \del_{(\check{\alpha},i)}: CF_*\to CF_*) &=& (N \del_{(\alpha_1,i)}: C_*\to C_*) \\&=& \frac{\del^2 \If_{S^1\times M}}{\del t^{\alpha_1,i-1}\del t^{\mu}} \eta^{\mu\nu} (N \del_{\nu}: C_*\to C_*) \\&+& (\frac{1}{2}[\del_{(\alpha_1,i-1)}, \check{N}\,\del]_+: C_*\to C_*) \\&+& (\frac{1}{2}[\del_{(\check{\alpha}_1,i-1)}, N\,\del]_+: C_*\to C_*)\\&=& \frac{\del^2 \If_M}{\del t^{\alpha,i-1}\del t^{\mu}} \eta^{\mu\nu} (N \del_{\check{\nu}}: CF_*\to CF_*) \\&+& (\frac{1}{2}[\del_{(\check{\alpha},i-1)}, \check{N}\,\del]_+: CF_*\to CF_*) , 
\end{eqnarray*}
where we further use that by similar arguments, namely that every holomorphic map $\CP\to\RS$ is constant, the Gromov-Witten potential of the stable Hamiltonian manifold $S^1\times M$ is given by the Gromov-Witten potential of the symplectic manifold $M$, see also the discussion at the beginning of subsection 4.2. \\

Apart from the fact that this proves the desired statement about the topological recursion relations, note that the same equation shows that in this special case the linear map $\del_{\check{\alpha}}: C_*\to C_*$ indeed leads to an action of the quantum cohomology of $S^1\times M$ on the \emph{equivariant} cylindrical homology, which agrees with the one defined by \cite{PSS} on symplectic Floer homology. For the action of quantum cohomology $QH^*(S^1\times M)$ on the non-equivariant cylindrical homology $HC^{\textrm{non-}S^1}_*(S^1\times M)$, observe that the differential of non-equivariant cylindrical homology is indeed of diagonal form \[\del = \diag(\del,\del): \hat{CF}_*\oplus \check{CF}_* \to  \hat{CF}_*\oplus \check{CF}_*\] with the Floer homology differential $\del: CF_*\to CF_*$. This follows from the fact that in this case the only off-diagonal contribution $\delta: \hat{CF}_*\to\check{CF}_*$ is also zero, which as above again follows from the fact that the $S^1$-symmetry on the cylinder is divided out twice. Furthermore note that by the same argument the linear map $\del_{\check{\alpha}}$ is also of diagonal form. It follows that the non-equivariant cylindrical homology is given as a direct sum, $HC^{\textrm{non-}S^1}_*=\hat{HF}_*\oplus\check{HF}_*$, that the quantum cohomology acts on both factors separately and agrees with the action defined in \cite{PSS}.
\end{proof}

\end{document}